\let\ORIlabel\label

\let\ORIrefstepcounter\refstepcounter

\AddToHook{package/hyperref/before}{%

   \let\label\ORIlabel

   \let\refstepcounter\ORIrefstepcounter

}
\documentclass[onefignum,onetabnum]{siamart171218}
\usepackage{todonotes}
\usepackage{algpseudocode}
\usepackage{ulem}

\usepackage{lipsum}
\usepackage{amsfonts}
\usepackage{bm}
\usepackage{graphicx}
\usepackage{epstopdf}
\usepackage{algorithmicx}
\usepackage{overpic}
\usepackage{rotating}
\usepackage{caption}
\usepackage{subcaption}
\ifpdf
  \DeclareGraphicsExtensions{.eps,.pdf,.png,.jpg}
\else
  \DeclareGraphicsExtensions{.eps}
\fi

\newsiamremark{remark}{Remark}
\newsiamremark{hypothesis}{Hypothesis}
\crefname{hypothesis}{Hypothesis}{Hypotheses}
\newsiamthm{claim}{Claim}

\headers{ALDI for rare event sampling}{D. Chakraborty, R. Harris, R. Klein,  et al}

\title{Affine Invariant Langevin Dynamics for rare event sampling\thanks{
\funding{The work has been funded by Deutsche Forschungsgemeinschaft (DFG) through
the grant CRC 1114 ‘Scaling Cascades in Complex Systems’ (project number 235221301,
project A02 - Multiscale data and asymptotic model assimilation for atmospheric flows).}}}

\author{
Deepmayan Chakraborty\thanks{Freie Universität Berlin (\email{deepyaman.chakraborty@fu-berlin.de}, \email{r.harris@fu-berlin.de}, \email{rupert.klein@math.fu-berlin.de}, 
\email{c.schillings@fu-berlin.de})}
\and Ruben Harris\footnotemark[2]
\and Rupert Klein\footnotemark[2]
\and Guillermo Olic\'on-M\'endez\thanks{Universität Potsdam (\email{guillermo.olicon@uni-potsdam.de}, \email{sereich@uni-potsdam.de}).}
\and Sebastian Reich\footnotemark[3]
\and Claudia Schillings\footnotemark[2]
}

\usepackage{amsopn}

\makeatletter
\newcommand*{\addFileDependency}[1]{
  \typeout{(#1)}
  \@addtofilelist{#1}
  \IfFileExists{#1}{}{\typeout{No file #1.}}
}
\makeatother

\usepackage{color}

\ifpdf
\hypersetup{
  pdftitle={Affine Invariant Langevin Dynamics for rare-event sampling},
  pdfauthor={Deepyaman Chakraborty, Ruben Harris, Rupert Klein, Guillermo Olic\'on-M\'endez, Sebastian Reich, Claudia Schillings}
}
\fi

\begin{document}

\maketitle


{\scriptsize \color{blue}To appear in \url{https://msp.org/soon/coming.php?jpath=camcos}}

\smallskip

\begin{abstract}
We introduce an affine invariant Langevin dynamics (ALDI) framework for the
efficient estimation of rare events in nonlinear dynamical systems. Rare events are formulated as Bayesian inverse
problems through a nonsmooth limit-state function whose zero level set characterises the
event of interest. To overcome the nondifferentiability of this function, we
propose a smooth approximation that preserves the failure set and yields a
posterior distribution satisfying the small-noise limit. The resulting potential
is sampled by ALDI, a (derivative-free) interacting particle system whose affine
invariance allows it to adapt to the local anisotropy of the posterior.

We demonstrate the performance of the method across a hierarchy of benchmarks, namely
two low-dimensional examples (an algebraic problem with convex geometry and a dynamical problem of saddle-type instability) and a point-vortex model for atmospheric blockings. In all cases, ALDI
concentrates near the relevant near-critical sets and provides accurate
proposal distributions for self-normalised importance sampling. The framework is
computationally robust, potentially gradient-free, and well-suited for complex forward
models with strong geometric anisotropy. These results highlight ALDI as a
promising tool for rare-event estimation in unstable regimes of dynamical systems.
\end{abstract}



\section{Introduction}
\label{sec:intro}
Rare events play a crucial role in many complex dynamical systems, ranging from transitions between metastable states in molecular dynamics \cite{KoltaiEtAl2016} to critical shifts in atmospheric and oceanic circulation regimes. In geophysical contexts, such events include persistent \emph{atmospheric blockings}---large-scale quasi-stationary flow patterns that interrupt the usual westerly circulation and can lead to prolonged periods of heat or precipitation extremes. Although these phenomena occur infrequently, their consequences can be severe, and quantifying their probability and underlying mechanisms remains a major scientific challenge \cite{DaviniEtAl2021,Hirt18,lucarini2019,Muller15,Ragone20, olicon24, olicon25}

The mathematical and computational treatment of rare events requires specialised tools, since direct numerical simulation becomes infeasible due to the exponentially small probability of observing such anomalies. Classical approaches include importance sampling and splitting methods \cite{Beck2017, cerou}, variational methods \cite{hartmann}, large deviation theory \cite{stadler} and particle methods \cite{delmoral}. More recently, a probabilistic reformulation of rare-event estimation as a \emph{Bayesian inverse problem} has been proposed \cite{Wagner21}, enabling the use of sampling and inference techniques from uncertainty quantification. From this perspective, the rare event is identified with the conditioning of a prior measure on the event $\{G(u)\le 0\}$, where $G$ is a suitably chosen observable. The resulting posterior is highly concentrated in a low-probability region, making it a challenging target for standard sampling algorithms.

Ensemble-based methods such as the Ensemble Kalman Filter (EnKF) have been widely used in data assimilation and inverse problems due to their computational efficiency and ability to handle high-dimensional state spaces \cite{evensen2009,reichcotter2015}. They propagate and update an ensemble of particles according to system dynamics and observational constraints, providing a derivative-free approximation of Bayesian updates. Ensemble-based methods such as the EnKF have proved effective far beyond the linear–Gaussian setting and are routinely applied to nonlinear, non-Gaussian problems \cite{evensen2003,evensen2009,reichcotter2015,vanleeuwen2010}. Their analysis step, however, is built around updating the first and second moments via a linear correction, which—together with finite-ensemble effects, localisation, and inflation—can bias tail probabilities and collapse separated modes unless special care is taken (e.g., tempering, iterative updates, or mixture/particle–ensemble hybrids) \cite{vanleeuwen2010,bocquet2010}. Consequently, representing an extremely small posterior mass in the tail region relevant for rare events remains challenging in practice, even though many EnKF variants alleviate parts of this issue.

To target rare events more directly, \cite{Wagner21} proposed an EnKF-based formulation in which the event is encoded through an observable $G$ (taken, without loss of generality, to be nonpositive on the event), and the inverse problem
\begin{equation}
  \tilde{G}(u)=0, 
  \qquad \text{with} \quad
  \tilde{G}(u):=\max\{0,G(u)\},
\end{equation}
so that the event corresponds to the set $\{G(u)\le 0\}$. This casts rare-event estimation as Bayesian conditioning on (a smoothed version of) the indicator of the rare set, thereby providing an inference route that we will retain here while replacing the ensemble update with an affine-invariant Langevin sampler (ALDI) \cite{Nusken20} as a stochastic, geometry-adapted sampling method for rare-event estimation. ALDI combines affine invariance, ensuring that the sampling dynamics is unaffected by linear rescaling or anisotropy of the state space, with the stochastic exploration of Langevin-type diffusions. This results in improved stability and efficiency, especially in high-dimensional or strongly anisotropic settings. The present work thus bridges ensemble-based inference, stochastic dynamics, and rare-event analysis within a unified geometric framework, building upon earlier developments in Bayesian sampling and dynamical systems theory. In \cite{beh2025affineinvariantinteractinglangevin}, a related method based on a combination of Langevin samplers and importance sampling (IS) has been introduced, which employs a smooth
approximation of the failure event and provides an analysis of both the resulting
importance sampling error and the discretization bias of the Langevin algorithm used for sampling the smoothed density. This forms the basis for an optimal computational work distribution between the Langevin sampler and the IS step. In contrast, our focus is on the small-noise limit of the
associated Bayesian formulation: we characterise the limiting posterior as the prior
restricted to the failure set and prove total variation convergence. This offers a
measure-theoretic foundation that is independent of any specific sampling scheme and
complements the algorithmic analysis contained in~\cite{beh2025affineinvariantinteractinglangevin}.

We apply ALDI to rare-event estimation in two simple test cases to analyse its behaviour, and then 
a point-vortex model for a two-dimensional flow, used to explore the persistence of quasistationary vortex configurations associated with atmospheric blockings \cite{Kuhlbrodt00}.
In both cases, we identify two qualitatively different mechanisms that generate low-probability events:
\begin{itemize}
    \item \textit{Parametric rare events}, where the event $\{G_\lambda(u)\le 0\}$ occurs only for specific parameter values $\lambda$, and
    \item \textit{Dynamic rare events}, where the event arises from the intrinsic instability of the system’s dynamics near invariant structures.
\end{itemize}

Understanding and quantifying such rare transitions is essential in many applications, as they often correspond to critical transitions or bifurcations in stochastic systems. Rare events are also related to bifurcations in stochastic dynamics when the noisy perturbations are bounded \cite{olicon24,olicon25}. Atmospheric blockings, for instance, represent prolonged excursions of the flow dynamics toward atypical attractors, the occurrence of which is often controlled by a delicate interplay between stochastic forcing, energy input, and multiscale interactions \cite{klein2021,ragone2021}.

\paragraph{Contributions and structure}
The main contributions of this paper are as follows:
\begin{itemize}
    \item Based on the reformulation of rare-event estimation as a Bayesian inverse problem, we introduce an ALDI method for exploring the associated conditional measure.
    \item We prove that, as the observation noise tends to zero, the estimator obtained from the ALDI method converges to the true probability of the rare event, independently of the value of the smoothing parameter.
    \item We demonstrate the performance of ALDI on dynamical systems around unstable invariant objects, including vortex configurations related to atmospheric blockings.
\end{itemize}
\paragraph{Relation to existing rare event sampling methods}
Rare event simulation has been extensively studied, and a variety of approaches have been proposed, including cross-entropy methods \cite{Rubinstein2001, PAPAIOANNOU2019106564}, subset simulation \cite{AU2001263}, adaptive importance sampling \cite{Botev}, Hamiltonian Monte Carlo \cite{PAPAKONSTANTINOU2023103485} and Langevin-based methods \cite{beh2025affineinvariantinteractinglangevin, cheng2025langevinbifidelityimportancesampling}, as well as ensemble-based approaches such as the Ensemble Kalman filter \cite{Wagner21} and consensus-based sampling \cite{doi:10.1137/23M1565966}. 

The approach considered here builds on the Bayesian inverse problem formulation of rare events introduced in \cite{Wagner21}, and combines it with affine-invariant Langevin dynamics as a sampling mechanism. A key additional ingredient is the introduction of a smoothed likelihood formulation of the rare event indicator, which enables the application of Langevin-type dynamics in this setting and provides a continuous approximation of the underlying rare event. 

Taken together, this results in a novel combination of a smoothed rare-event formulation with an ensemble-based, affine-invariant stochastic dynamics, which can be interpreted as a Markov chain importance sampling strategy. Rather than aiming to replace existing rare event sampling methods, the proposed approach provides a complementary perspective that is particularly attractive in settings where ensemble-based formulations, affine invariance, or derivative-free implementations are advantageous.

The remainder of this paper is structured as follows.
Section~\ref{SEC:preliminaries} reviews both the Bayesian formulation of rare-event estimation and the affine-invariant Langevin dynamics. To exploit the theoretical features of the latter, we introduce a smoothing technique used in this context, along theoretical results concerning the invariance and convergence of the method. In Section~\ref{SEC:numerics}, we test the method on two simple problems: an algebraic convex example (Subsection~\ref{SUBSEC:convex}) and a linear unstable dynamical system (Subsection~\ref{SUBSEC:saddle}). We further exploit the method for rare-event estimation in point-vortex dynamics (Subsection~\ref{SUBSEC:vortex}). Finally, Section~\ref{SEC:conclusions} summarises the results and discusses perspectives for future research.


\section{Preliminaries}
\label{SEC:preliminaries}
In this paper, we consider rare events defined by the outcome of a \textit{limit-state function} $G$. More precisely, let $(\Omega,\mathcal{F},\mathbb{P})$ be a probability space, and the state of a system given by a random variable $X:\Omega\rightarrow \mathbb{X}$, where the state space $\mathbb{X}=\mathbb R^d$ endowed with its Borel $\sigma$-algebra $\mathcal{B}$, and distributed according to $\mu_0(\mathrm d x)=\rho_0(x) \mathrm dx$. 

We will be interested in the convergence of a family of measures $\mu_\delta$ on $(\mathbb X, \mathcal{B})$ to a given measure $\mu$. Recall that the \emph{total variation distance} between probability distributions $\mu_1,\mu_2$ is defined as $\Vert \mu_1-\mu_2\Vert_{TV}:= \sup_{A\in\mathcal{B}} \vert \mu_1(A)-\mu_2(A)\vert$. Moreover, when both distributions are absolutely continuous with respect to the same distribution $\mu_0$, so that $\mu_i(dx)=\rho_i(x)\mu_0(dx)$, then
\[\Vert \mu_1-\mu_2\Vert_{TV}\ = \frac{1}{2}\Vert \rho_1-\rho_2\Vert_{L^1(\mu_0)}:= \frac{1}{2}\int_{\mathbb X} \vert \rho_1-\rho_2\vert d\mu_0. \]
Given a density $\rho$ of a distribution $\mu$, we will sometimes refer to the Lebesgue space $L^1(\\mu)\equiv L^1(\rho)$.

\subsection{Rare events as Bayesian inverse problems}

For a given \emph{limit-state function} \( G : \mathbb{X} \rightarrow \mathbb{R} \), the \emph{failure domain} is defined (without loss of generality) by
\[
F=\{ G(x) \le 0 \}.
\]
We are interested in estimating the associated probability of failure
\begin{equation}
    \label{eq:probfailure_event}
    P_f := \mathbb{P}\left( \left\{ \omega \in \Omega : G(X(\omega)) \le 0 \right\} \right)=\int_\mathbb{X} \mathbf{1}_{F}(x) \mu_0(\mathrm dx)
\end{equation}
where $\mathbf 1 $ denotes the indicator function and $\mu_0$ a given probability measure on $\mathbb X$. 
In many situations, the failure domain is highly nonlinear, disconnected, or narrow, and the corresponding event is rare. Consequently, estimating \( P_f \) accurately requires specialised methods that go beyond standard Monte Carlo techniques, which tend to undersample low-probability regions. One fruitful approach is to recast the identification of failure states as an \emph{inverse problem}. 

To this end, following \cite{Wagner21}, we introduce the modified limit-state function
\[
\tilde{G}(x) := \max\{ 0, G(x) \},
\]
and consider the inverse problem
\begin{equation}
    \label{eq:inverse}
    \tilde{G}(x) = 0,
\end{equation}
whose solution set coincides with the failure domain $\{G(x)\le 0\}$.

We adopt a Bayesian formulation by introducing an artificial observation model with additive noise,
\begin{equation}
    \label{eq:bayesian_inverseproblem}
    y = \tilde{G}(x) + \eta,
\end{equation}
where \( \eta \sim \mathcal{N}(0, R) \) with variance parameter \( R > 0 \), and we fix the observed value to \( y = 0 \). The noise variance \( R \) acts as a \emph{smoothing parameter}: as \( R \to 0 \), the posterior concentrates near the zero level set of \( \tilde{G} \), i.e., the boundary of the failure region.

Denoting by \( \rho_0(x) \) the density of $\mu_0$, Bayes’ theorem yields the posterior density
\begin{equation}
    \label{eq:posterior distribution}
    \rho_*(x) = \frac{1}{Z} \exp\!\left(-\frac{1}{2R}\tilde{G}(x)^2\right)\rho_0(x),
\end{equation}
where \( Z \) is the normalising constant. The posterior distribution is in general non-Gaussian, since \( \tilde{G} \) is typically nonlinear and may only be continuous. In the following we assume that the prior density $\rho_0$ is strictly positive
and continuous on $\mathbb X$, and that the limit-state function $G:\mathbb X\to\mathbb R$ is
continuous. Consequently, $\tilde G(x)=\max\{0,G(x)\}$ is also continuous but
not differentiable on the boundary $\{G(x)=0\}$, which motivates the
smoothing procedure described in Section~\ref{subsubsec:smooth}.

In \cite{Wagner21}, a variant of the EnKF was employed to approximate the failure probability \(P_f\) based on this Bayesian reformulation. In the present work, we replace the EnKF by ALDI, whose invariant measure coincides with the smoothed posterior \eqref{eq:posterior distribution}. Unlike the EnKF, which relies on affine ensemble updates and provides a Gaussian approximation, ALDI preserves the full non-Gaussian structure of the posterior and remains consistent in the small-noise and vanishing-smoothing limits.

\subsection{Affine-Invariant Interacting Langevin Dynamics}\label{subsec:ALDI}

In our numerical experiments, we employ ALDI \cite{Nusken20} to sample the posterior distribution associated with the Bayesian inverse problem \eqref{eq:bayesian_inverseproblem}. In general, ALDI samples from a probability density of the form
\[
\rho_*(x) = \frac{1}{Z}\exp(-\Phi(x)),
\]
where \(Z>0\) is a normalisation constant and \(\Phi:\mathbb{R}^d\to\mathbb{R}\) denotes the \emph{potential function}.  
In our setting, the potential function reads
\begin{equation}
    \label{eq:potential_function}
    \Phi(x) = \frac{1}{2R}\,\tilde{G}(x)^2 - \ln \rho_0(x),
\end{equation}
so that the distribution corresponding to $\rho_*$ is absolutely continuous with respect to the prior measure \(\mu_0\), with density proportional to \(\exp(-\tfrac{1}{2R}\tilde{G}(x)^2)\). We assume $\rho_0$ to be strictly positive and smooth enough so that $-\ln \rho_0$ is well-defined and differentiable where needed.

The ALDI method is formulated as an interacting particle system.  
If the potential function \(\Phi\) is smooth, the dynamics of an ensemble of \(J\) particles \(X=(x^{(1)},\ldots,x^{(J)})\) is given by the system of Itô-type stochastic differential equations (SDE)
\begin{equation}
    \label{eq:ALDI_smooth}
    dx_t^{(j)} 
    = \left[-\mathcal{C}(X_t)\,\nabla_{x^{(j)}} \Phi\!\left(x_t^{(j)}\right)
    + \frac{d+1}{J}\big(x_t^{(j)} - m(X_t)\big)\right] dt
    + \sqrt{2}\,\mathcal{C}^{1/2}(X_t)\,dW_t^{(j)},
\end{equation}
where \(W_t^{(j)}\) are independent \(d\)-dimensional standard Brownian motions.  
Here, \(m(X)\) and \(\mathcal{C}(X)\) denote the empirical mean and covariance of the ensemble,
\begin{equation}
    \label{eq:ALDI_definitions}
    \begin{split}
        m(X) &:= \frac{1}{J}\sum_{j=1}^J x^{(j)}, \\[2mm]
        \mathcal{C}(X) &:= \frac{1}{J}\sum_{j=1}^J \big(x^{(j)}-m(X)\big)\big(x^{(j)}-m(X)\big)^\top, \\[2mm]
        \mathcal{C}^{1/2}(X) &:= \frac{1}{\sqrt{J}}\big(X - m(X)1_J\big),
    \end{split}
\end{equation}
with \(1_J=(1,\ldots,1)^\top \in \mathbb{R}^J\).

Let \(\pi_t^{(J)}\) denote the probability density of \(X_t\).  
The corresponding extended invariant density is
\begin{equation}
    \label{eq:invariant_density}
    \pi_*^{(J)}(x) = \prod_{j=1}^J \rho_*\!\left(x^{(j)}\right).
\end{equation}

\begin{theorem}
\label{THM:ALDI_properties}
Let $\mathbb X=\mathbb{R}^d$. Assume that the initial covariance \(\mathcal{C}(X_0)\) is positive definite, the potential \(\Phi \in C^2(\mathbb{R}^d)\cap L^1(\rho_*)\), and that there exist a compact set \(K \subset \mathbb{R}^d\) and constants \(0 < c_1 < c_2\) such that for all \(x \in \mathbb{R}^d \setminus K\),
\begin{equation}
    \label{eq:conditions_potential}
    \begin{split}
        c_1 \|x\|^2 \le \Phi(x) \le c_2 \|x\|^2, \quad
        c_1 \|x\| \le \|\nabla \Phi(x)\| \le c_2 \|x\|, \quad
        c_1 I_d \le \mathrm{Hess}\,\Phi(x) \le c_2 I_d.
    \end{split}
\end{equation}
Then the ALDI system \eqref{eq:ALDI_smooth} satisfies the following properties:

\smallskip
\noindent
1. \textbf{(Affine invariance)}  
The nonlinear Fokker–Planck equation associated with \eqref{eq:ALDI_smooth},
\begin{equation}
\label{eq:Fokker_Planck_nonlinear}
\partial_t\pi_t = 
\nabla_x \!\cdot\!
\left(
\pi_t\,C(\pi_t)\,\nabla_x
\frac{\partial \mathrm{KL}(\pi_t\Vert \rho_*)}{\partial \pi_t}
\right),
\quad
C(\pi_t) = \mathbb{E}_{\pi_t}\!\left[(x-\mu_t)(x-\mu_t)^\top\right],
\end{equation}
is affine invariant.

\smallskip
\noindent
2. \textbf{(Invariant distribution)}  
The product measure \(d\mu(x)=\pi_*^{(J)}(x)\,dx\) is invariant for the particle system \eqref{eq:ALDI_smooth}.

\smallskip
\noindent
3. \textbf{(Global strong solution)}  
The system \eqref{eq:ALDI_smooth} admits a unique global strong solution \(X_t\), and the empirical covariance \(\mathcal{C}(X_t)\) remains strictly positive definite almost surely.

\smallskip
\noindent
4. \textbf{(Ergodicity)}  
If \(J > d+1\), then \(\pi_t^{(J)} \to \pi_*^{(J)}\) in total variation distance as \(t \to \infty\).

\end{theorem}

\begin{proof}
We refer to Lemma~2.4, Corollary~4.2, Proposition~4.4, Proposition~4.5, and Lemma~4.7 in \cite{Nusken20}.
\end{proof}

\begin{remark}
The ALDI method considered in this work should be understood as one possible sampling strategy within the Bayesian inverse problem formulation of rare events. Other sampling approaches, such as Metropolis-adjusted Langevin algorithms, Hamiltonian Monte Carlo, or adaptive importance sampling techniques, could in principle be used in the same framework. 
\end{remark}
\subsubsection{A smoothing of the limit-state function}\label{subsubsec:smooth}
As mentioned earlier, our problem involves a potential $\Phi$ that is nonsmooth. In order to exploit the theoretical aspects of the ALDI method \eqref{eq:ALDI_smooth} as given in Theorem~\ref{THM:ALDI_properties}, we introduce a smooth approximation of $\tilde{G}$ which coincides with it for any value of $G(x)$ outside a small interval $(0,\delta)$. Consider the function
\begin{equation}
\label{eq:psi_base}
\psi_\delta(x):=\begin{cases}
\exp\left(\frac{1}{\delta^2}\right)\cdot \exp\left(-\frac{1}{x^2}\right), & \textup{if }
x> 0\\
0, & \textup{otherwise}, \end{cases}
\end{equation}
and the ramp function
\begin{equation}
    \label{eq:ramp_function}
\phi_\delta(x):=\frac{\psi_\delta(x)}{\psi_\delta(x)+ \psi_\delta(\delta-x)},
\end{equation}
which approximates smoothly the indicator function $\mathbf{1}_{[0,\infty)}$. Therefore, a smooth approximation of $\tilde{G}$ is given by
\begin{equation}
    \label{eq:smooth_G}
\tilde G_\delta(x) :=
\begin{cases}
  0, & \text{if } G(x) < 0,\\[0.3em]
  G(x)\dfrac{\psi_\delta(G(x))}{\psi_\delta(G(x))+\psi_\delta(\delta-G(x))}, & \text{if } G(x)\in[0,\delta],\\[0.7em]
  G(x), & \text{if } G(x) > \delta.
\end{cases}
\end{equation}
Note that $\tilde G_\delta(x)\to\tilde G(x)$ uniformly as $\delta\to0$, and the zero level set $\{\tilde{G}_\delta=0\}$ is preserved for all $\delta\geq0$. For the theoretical results establishing convergence of ALDI, we assume
that \(G\in C^2(\mathbb X)\) to ensure the smoothness requirements of the potential. Since the scalar transition function used in
\eqref{eq:smooth_G} is \(C^\infty\) and flat at the interfaces \(G=0\) and
\(G=\delta\), the smoothed limit-state function
\(\widetilde G_\delta\) is \(C^2\). Consequently,
\[
    \Phi_{\delta,R}(x)
    =
    \frac{1}{2R}\widetilde G_\delta(x)^2-\log\rho_0(x)
\]
is \(C^2\), provided \(\rho_0\) is strictly positive and
\(-\log\rho_0\in C^2\).
\subsubsection{Consistency of the smoothed Posterior and ALDI}
\label{SUBSUBSEC:consistency}
We now establish that the smoothed posterior distribution underlying the ALDI dynamics 
converges, in the joint limit of vanishing observation noise and vanishing smoothing, 
to the prior distribution restricted to the rare-event domain.

\begin{theorem}[Consistency of the smoothed posterior]
\label{thm:consistency}
Let $F=\{x\in\mathbb X : G(x)\le 0\}$ be the failure domain and assume that $P_f=\mu_0(F)>0$. Let $\tilde G_\delta$ be the smoothed approximation defined in
\eqref{eq:smooth_G} and, for each $R>0$, define the probability measures
\begin{equation}\label{eq:smoothedpost}
  \mu_{\delta,R}(dx)
  := \frac{1}{Z_{\delta,R}}
     \exp\!\left(-\frac{\tilde G_\delta(x)^2}{2R}\right)\rho_0(x)\,dx,
  \quad  
  Z_{\delta,R}
  := \int_{\mathbb X} \exp\!\left(-\frac{\tilde G_\delta(x)^2}{2R}\right)
      \rho_0(x)\,dx.
\end{equation}
Then, for every fixed $\delta>0$ we have $\mu_{\delta,R}\;\xrightarrow{\mathrm{TV}}\;\mu_F$ as $R\rightarrow 0 $,
where the distribution $\mu_F(dx)=\frac{\mathbf{1}_F(x)}{\mu_0(F)}\rho_0(x)dx$ is independent of~$\delta$. 
\end{theorem}

\begin{proof}

Consider $\delta>0$ fixed. First, we show that the normalisation constants $Z_{\delta,R}$ converge to $\mu_0(F)$. Indeed, recall from \eqref{eq:smooth_G} that $\tilde{G}_\delta(x)=0$ whenever $x\in F$, for all $\delta\geq 0$. By splitting the state space $\mathbb{X}=F \cup (\mathbb X \setminus F)$, it follows from \eqref{eq:smoothedpost} that
\[
    Z_{\delta,R} = \mu_0(F) + \int_{\mathbb X \setminus F} \exp\left( -\frac{\tilde{G}_\delta(x)^2}{2R} \right) \rho_0(x)dx.
\]
The integrand on the right hand side is bounded by $\rho_0(x)$, and $\exp(-\tilde{G}(x)^2/(2R))\rightarrow 0$ for each $x\in \mathbb{X}\setminus F$. Hence, by the dominated convergence theorem we obtain that
\begin{equation}
    \label{eq:limit_normalisation}
        \lim_{R\rightarrow 0} Z_{\delta,R}=\mu_0(F).
\end{equation}
We denote
\[
N_{\delta,R}(x)
:=
\exp\!\left(-\frac{\tilde G_\delta(x)^2}{2R}\right).
\]
and $f_{\delta,R}(x)
:=
\frac{N_{\delta,R}(x)}{Z_{\delta,R}}$ the density of $\mu_{\delta,R}$ with respect to $\mu_0$. We show that $f_{\delta,R}\rightarrow f_F(x):= \frac{1}{\mu_0(F)}\mathbf{1}_F(x)$ in $L^1(\mu_0)$. Indeed, by splitting again the state space and recalling that $\tilde G(x)=0$ for $x\in F$, we get
\begin{align*}
    \int_{\mathbb X}\vert f_{\delta,R}-f_F \vert d\mu_0 & = \int_F \left\vert \frac{N_{\delta,R}(x)}{Z_{\delta,R}} - \frac{1}{\mu_0(F)}\right\vert d\mu_0(x) + \int_{\mathbb X \setminus F} \frac{N_{\delta, R}(x)}{Z_{\delta,R}}d\mu_0(x)= \\
    &= \frac{\left\vert Z_{\delta,R}-\mu_0(F) \right\vert}{Z_{\delta,R}}+ \frac{1}{Z_{\delta,R}}\int_{\mathbb{X}\setminus F}N_{\delta,R}(x)d\mu_0(x).
\end{align*}
As $R\rightarrow 0$, the first term on the right hand side vanishes as given in \eqref{eq:limit_normalisation}. For the second term, observe that $N_{\delta,R}\rightarrow \mathbf{1}_F$ pointwise and that $N_{\delta,R}\leq 1 \in L^1(\mu_0)$, so that the integral on the right hand side vanishes as $R\rightarrow 0$ due to the dominated convergence theorem. 

Since the densities $f_{\delta,R}$ converge to $f_F$ in $L^1(\mu_0)$, their corresponding measures $\mu_{\delta,R}$ converge to $\mu_F$ in total variation as they are all absolutely continuous with respect to $\mu_0$, and the result follows.
\end{proof}

\begin{remark}
The proof of Theorem~~\ref{thm:consistency} relies only on the continuity of $G$, 
the property that $\tilde G_\delta(x)=0$ if and only if $G(x)\le 0$.
In particular, for every fixed $\delta>0$ we have
\[
\mu_{\delta,R} \xrightarrow[R\to 0]{\mathrm{TV}} \mu_F,
\]
and the limit measure $\mu_F$ is independent of $\delta$. 
Thus, $\delta$ acts purely as a smoothing parameter that improves the 
regularity of the potential but does not affect the limiting conditional
distribution as $R\to 0$. 
The results itself is formulated at the level of the target measures 
$\mu_{\delta,R}$ and is therefore independent of the particular sampling 
scheme. In particular, any algorithm whose invariant distribution is 
$\mu_{\delta,R}$ (such as the gradient-based ALDI dynamics) inherits the same consistency property.
\end{remark}

Although the preceding result already identifies the limiting measure for fixed
$\delta$, we record the following joint-limit consequence to make explicit
that the same target measure is recovered when the smoothing parameter and the
observation noise vanish simultaneously.

\begin{corollary}[Joint limit]
\label{cor:joint-limit}
Assume that \(0<\mu_0(F)<1\). Then, for any sequence
\((\delta_n,R_n)\) with \(\delta_n\to0\) and \(R_n\to0\),
\[
    \mu_{\delta_n,R_n}\to \mu^\ast
    \qquad \text{in total variation}.
\]
\end{corollary}

\begin{proof}
By the definition of the smoothed likelihood, for every $x\in \mathbb X$,
\[
N_{\delta_n,R_n}(x)
=
\exp\!\left(-\frac{\widetilde G_{\delta_n}(x)^2}{2R_n}\right)
\to
\mathbf 1_{\{G(x)\le0\}}.
\]
Indeed, if \(G(x)\le0\), then \(\widetilde G_{\delta_n}(x)=0\) for all $n\in\mathbb N$, whereas if
\(G(x)>0\), then \(\widetilde G_{\delta_n}(x)=G(x)\) for all sufficiently
large \(n\). Since \(0\le N_{\delta_n,R_n}\le1\), dominated convergence
implies convergence of the normalizing constants and of the normalized
densities in \(L^1(\mu_0)\), which is equivalent to convergence in total
variation.
\end{proof}

Theorem~\ref{THM:ALDI_properties} can be translated in terms of the ALDI method as given in the following statement.

\begin{corollary}[Consistency of ALDI]
Assume that $G \in C^2(\mathbb{X})$ and that $\rho_0$ is strictly positive with $-\log \rho_0 \in C^2(\mathbb{X})$. Then the smoothed potential $\Phi_{\delta,R}(x) = \frac{1}{2R}\tilde{G}_\delta(x)^2 - \log \rho_0(x)$ satisfies $\Phi_{\delta,R} \in C^2(\mathbb{X})$ for every $\delta > 0$, and the assumptions of Theorem \ref{THM:ALDI_properties}  apply.
For every fixed smoothing parameter $\delta>0$ and every noise level $R>0$,
the ALDI dynamics \eqref{eq:ALDI_smooth} admits the invariant product measure
\[
\pi^{(J),\delta,R}_{*}(x) \;\propto\; 
\prod_{j=1}^J 
\exp\!\left(-\frac{ G_\delta(x^{(j)})^2 }{2R}\right)\rho_0(x^{(j)}),
\]
and is ergodic whenever $J>d+1$ (see Theorem~\ref{THM:ALDI_properties}).  
If the dynamics is run to stationarity, each particle satisfies
\[
X^{(j)}_{\delta,R}\sim \rho^*_{\delta,R},
\]
where $\rho^*_{\delta,R}=
\frac{1}{Z_{\delta,R}}
\exp\!\left(-\frac{\tilde G_\delta(x)^2}{2R}\right)
$ is the smoothed posterior defined.

Moreover, by Theorem~\ref{thm:consistency}, for every fixed $\delta>0$,
\[
\rho^*_{\delta,R} 
\;\xrightarrow[R\to 0]{\rm TV}\;
\mu_F ,
\qquad 
F=\{x : G(x)\le 0\},
\]
and the limit measure $\mu_F$ is independent of $\delta$.
Consequently, the stationary ALDI estimator is consistent for the
rare-event distribution: for each ensemble index $j$,
\[
X^{(j)}_{\delta,R}
\;\xrightarrow[R\to 0]{\rm law}\;
\mu_F .
\]
\end{corollary}

\begin{remark}
When the prior is Gaussian, i.e., $\rho_0(x) \propto \exp\left(-\frac{1}{2}x^\top P_0^{-1} x\right)$, the term $-\log \rho_0(x)$ is quadratic and dominates the bounded contribution from the term $\tilde{G}_\delta(x)^2/(2R)$ near the failure set. In particular, the growth conditions (2.9) are satisfied provided $G$ has at most polynomial growth at infinity. All numerical examples in the following use Gaussian priors and satisfy these conditions.
\end{remark}

 \subsection{Importance sampling for the probability of failure}
\label{subsec:IS}

Recall that our goal is to estimate the probability of failure
\[
  P_f = \mu_0(F)
      = \int_X \mathbf{1}_{\{G(x)\le 0\}}\,\rho_0(x)\,dx,
  \qquad F := \{x\in X : G(x)\le 0\}.
\]
A standard approach is IS, which rewrites
\[
  P_f
  = \int_F \frac{\rho_0(x)}{q(x)}\, q(x)\,dx,
\]
where $q$ is a proposal density. The usual global support condition
$\operatorname{supp}(\rho_0)\subseteq \operatorname{supp}(q)$ is not required here.
Since the integrand contains $\mathbf{1}_{F}$, the correct condition is only
\begin{equation}
  \label{eq:support-cond-F}
  \operatorname{supp}\bigl(\mathbf{1}_{F}\rho_0\bigr)
    \subseteq \operatorname{supp}(q),
\end{equation}
i.e., $q(x)>0$ for all $x\in F$ with $\rho_0(x)>0$.
This is naturally satisfied for the constructions below. A central difficulty in importance sampling is \emph{weight degeneracy}, where only a small number of samples carry non-negligible weights due to a mismatch between the proposal distribution and the rare-event region, resulting in large estimator variance and a severely reduced effective sample size. In particular, importance sampling based on the prior as proposal is highly prone to weight degeneracy in rare-event regimes, whereas ALDI-based proposals concentrate mass near the failure set and therefore yield substantially more balanced weights.

We consider two IS strategies based on stationary samples from
$\rho^*_{\delta,R}$, an auxiliary smoothed density used to bias sampling
towards the failure set.
Both methods are consistent for estimating the failure probability $P_f$,
but differ in how the proposal distribution is constructed and exploited.

\paragraph{Variant~1: Product estimator based on ALDI samples.}
Assume that the samples $\{X_m\}_{m=1}^J$ are obtained from ALDI at stationarity,
so that $X_m\sim \rho^*_{\delta,R}$, where
\[
  \rho^*_{\delta,R}(x)
  = \frac{1}{Z^*_{\delta,R}}\,
    \exp\!\Bigl(-\tfrac{1}{2R}\tilde G_\delta(x)^2\Bigr)\,\rho_0(x),
\]
and
\[
  Z^*_{\delta,R}
  := \int \exp\!\Bigl(-\tfrac{1}{2R}\tilde G_\delta(x)^2\Bigr)\,\rho_0(x)\,dx .
\]

Since $\tilde G_\delta(x)=0$ for all $x\in F=\{G(x)\le 0\}$, the failure
probability can be written as
\[
P_f
= \int_F \rho_0(x)\,dx
= Z^*_{\delta,R}\,
  \mathbb P_{\rho^*_{\delta,R}}(X\in F).
\]
This identity motivates the product estimator
\[
  \widehat P_f
  = \widehat Z^*_{\delta,R}\,
    \widehat p^*,
  \qquad
  \widehat p^*
  := \frac{1}{J}\sum_{m=1}^J
      \mathbf{1}_{\{G(X_m)\le 0\}},
\]
where $\widehat p^*$ is the empirical probability of the failure event under
$\rho^*_{\delta,R}$, and $\widehat Z^*_{\delta,R}$ is an estimator of the
normalisation constant. Since $\rho^*_{\delta,R}(x)=\frac{1}{Z^*_{\delta,R}}\exp\!\big(-\tfrac{1}{2R}\tilde G_\delta(x)^2\big)\,\rho_0(x)$, we can rewrite
\[
\rho_0(x)=Z^*_{\delta,R}\,\exp\!\big(\tfrac{1}{2R}\tilde G_\delta(x)^2\big)\,\rho^*_{\delta,R}(x).
\]
Integrating both sides over $x$ yields the inverse-moment identity
\[
1 = Z^*_{\delta,R}\,\mathbb E_{\rho^*_{\delta,R}}\!\left[\exp\!\big(\tfrac{1}{2R}\tilde G_\delta(X)^2\big)\right],
\qquad\Rightarrow\qquad
Z^*_{\delta,R}=\left(\mathbb E_{\rho^*_{\delta,R}}\!\left[\exp\!\big(\tfrac{1}{2R}\tilde G_\delta(X)^2\big)\right]\right)^{-1},
\]
which motivates the Monte Carlo approximation in Alg.~\ref{alg:aldi_product} by replacing the expectation with an empirical average over stationary ALDI samples.

\begin{algorithm}[t]
\caption{ALDI-based product estimator for $P_f$}
\label{alg:aldi_product}
\begin{algorithmic}[1]
\Require prior $\rho_0$, limit-state $G$, smoothing parameters $\delta,R$,
         ALDI integrator, ensemble size $J$, sample size $M$
\Ensure estimate $\widehat P_f$
\State initialise ALDI ensemble
\State run ALDI (possibly with annealing) towards $\rho^*_{\delta,R}$
\State collect $M$ approximately independent samples $\{X_m\}$
\State compute
       $\displaystyle
        \widehat p^*=
        \frac{1}{M}\sum_{m=1}^M \mathbf{1}_{\{G\le 0\}}(X_m)$
\State estimate normalisation constant $\widehat Z^*_{\delta,R}
=
\left(
\frac{1}{M}\sum_{m=1}^M
\exp\!\Bigl(\tfrac{1}{2R}\tilde G_\delta(X_m)^2\Bigr)
\right)^{-1}.
$
       
\State compute $\widehat P_f=\widehat Z^*_{\delta,R}\,\widehat p^*$
\State \Return $\widehat P_f$
\end{algorithmic}
\end{algorithm}

\begin{remark}
Alg.~\ref{alg:aldi_product} is included for conceptual completeness. We will see in the numerical experiments that the normalisation estimator is numerically unstable, as given in Figure~\ref{fig:size_and_failed}(a).
\end{remark}

\paragraph{Variant~2: Mixture-based IS using a fitted proposal.}
To increase the accuracy of the estimator to a prescribed tolerance, one
typically needs to generate a large number of independent samples from the
proposal distribution.  
To make this feasible, we fit a flexible proposal
$q_\theta$, e.g.\ a Gaussian mixture, to a set of ALDI samples
and subsequently perform IS with respect to $q_\theta$. Close to stationarity, the resulting mixture will satisfy the support condition~\eqref{eq:support-cond-F}, since
$\rho^*_{\delta,R}$ concentrates near the failure set $F$ for small
$(\delta,R)$ and assigns positive mass to all regions in $F$ where
$\rho_0$ is positive.  

Given $X_m\sim q_\theta$, the (unnormalised) IS weights are
\[
  w_m := \frac{\rho_0(X_m)}{q_\theta(X_m)},
\]
and the resulting estimator is again
\[
  \widehat P_f^{\mathrm{IS}}
  = \frac{1}{M}\sum_{m=1}^M\mathbf{1}_{\{G(X_m)\le 0\}}\,w_m.
\]
This variant has the additional advantage that, once the proposal $q_\theta$ has
been fitted, an arbitrary number of independent samples can be generated at very
low cost. This enables the IS estimator to achieve any prescribed accuracy
by drawing sufficiently many proposal samples. However, additional forward solves are thus needed to evaluate the IS estimator. 
Moreover, the separation between the ALDI phase (which generates the fitting
data) and the IS phase (which uses $q_\theta$) suggests that the overall work
can be balanced: ALDI does not need to be run until full convergence. As soon as
the ALDI cloud captures the relevant geometry of the failure region, the
subsequent IS step can be made as accurate as desired by increasing the number
of independent samples from $q_\theta$. This opens up the possibility of stopping
ALDI earlier while still reaching a target accuracy for $P_f$ at reduced
computational cost.

\begin{algorithm}[t]
\caption{Mixture-based IS fitted to ALDI samples}
\label{alg:aldi_mixture_is}
\begin{algorithmic}[1]
\Require prior $\rho_0$, limit-state $G$, smoothing parameters $\delta,R$,
         ALDI integrator, ensemble size $J$,
        IS sample size $M$,
         mixture size $K$
\Ensure estimate $\widehat P_f$ and effective sample size $\mathrm{ESS}$
\State run ALDI to stationarity and collect $M_{\mathrm{fit}}$ samples
       $\{X^{\mathrm{ALDI}}_m\}$
\State fit a parametric proposal $q_\theta$
       (e.g.\ a $K$-component Gaussian mixture) to the ALDI cloud
\State sample $X_1,\dots,X_M\sim q_\theta$
\For{$m=1,\dots,M$}
  \State $I_m:=\mathbf{1}_{\{G(X_m)\le 0\}}$
  \State $w_m:=\rho_0(X_m)/q_\theta(X_m)$
\EndFor
\State estimate
       $\widehat P _f^{\mathrm{IS}}= \frac{1}{M}\sum_{m=1}^MI_m w_m$
\State \Return $\widehat P_f^{\mathrm{IS}}$
\end{algorithmic}
\end{algorithm}

\subsection{Gradient-free ALDI method}
As an alternative formulation, a gradient-free version of the ALDI method can be used \cite{Nusken20}. In this subsection we assume a Gaussian prior
\[
\mu_0=\mathcal N(m_0,P_0),
\]
where \(m_0\) and \(P_0\) denote the prior mean and covariance, respectively. One considers the sample cross-correlation matrix $\mathcal{D}(U)$ defined as
\begin{equation}
    \label{eq:crosscorrelation_matrix}
\mathcal{D}(X):= \frac{1}{J}\sum_{j=1}^J \left(x^{(j)}-m(X)\right) \left( \mathcal{G}(x^{(j)})-m(\mathcal{G}(X)) \right)^\top,
\end{equation}
where
\begin{equation}
    \label{eq:sample_mean2}
m(\mathcal{G}(X)):= \frac{1}{J}\sum_{j=1}^J \mathcal{G}\left(x^{(j)} \right),
\end{equation}
so that the \textit{gradient-free ALDI method} reads as
\begin{equation}
    \label{eq:ALDI_gradientfree}
\begin{split}
    dx^{(j)}_t & = -\left[ \mathcal{D}(X_t) R^{-1}\left( \mathcal{G}\left( x_t^{(j)} \right) -y_{obs}\right) +\mathcal{C}(X_t)P_0^{-1}\left(  x_t^{(j)}-m_0\right) \right]dt \\
    & \qquad + \frac{d+1}{J}\left(  x_t^{(j)}-m(X_t)\right)dt + \sqrt{2} \mathcal{C}^{1/2}(X_t)dW^{(j)}_t,
\end{split}
\end{equation}
Here \(y_{\mathrm{obs}}\) denotes the observed value in the inverse problem formulation. In the rare-event setting considered here, we take \(y_{\mathrm{obs}}=0\).
\begin{remark}
    In our case, the forward map $\mathcal{G}=\tilde{G}$ and the observational noise $\eta$ are real-valued. Moreover, $y_{obs}=0$, which simplifies the computations. In fact,
\begin{equation}        \label{eq:ALDI_gradientfree_particular}
\begin{split}
dx_t^{(j)} & = -\left[ \frac{1}{R}\mathcal{D}(X_t)\tilde{G}\left(x_t^{(j)}\right) + \frac{1}{P_0}\mathcal{C}(X_t)\left( x_t^{(j)}-m_0\right) + \frac{d+1}{J}\left( x_t^{(j)}-m(X_t) \right)\right] dt \\
 & \qquad \qquad +\sqrt{2} \mathcal{C}^{1/2}(X_t)dW_t^{j}.
\end{split}
\end{equation}
\end{remark}

As shown in \cite{Nusken20}, the system \eqref{eq:ALDI_gradientfree} is also affine invariant, like its exact counterpart \eqref{eq:ALDI_gradientfree}. The drawback, however, is that a theoretical analysis as detailed in Theorem~\ref{THM:ALDI_properties} remains still an open problem.
The gradient-free variant relies on the Gaussian prior structure to construct
ensemble-based approximations of the gradient. While this approach can be
applied more broadly in practice, a rigorous theoretical analysis, in particular
for non-Gaussian priors, is currently not available and remains an interesting
direction for future work.

\section{Illustrative Test Cases}
\label{SEC:numerics}

In this section, we assess the validity of the ALDI method for rare-event estimation by means of three illustrative examples. We first consider a classical example from structural reliability, introduced in
\cite{Katsuki94}. Secondly, we explore a simple linear ordinary differential equation with an unstable equilibrium of saddle type. Finally, we explore a dynamical model of interacting point-vortices.

\smallskip

\paragraph{Discussion of comparisons}
 In this work, our main focus is on demonstrating the feasibility of combining the Bayesian inverse problem formulation of rare events with affine-invariant Langevin dynamics. Nevertheless, we briefly comment on the relation to other approaches.  Compared to cross-entropy and adaptive importance sampling methods, ALDI evolves an interacting particle system that approximates the posterior distribution without explicitly constructing a parametric biasing distribution. In contrast to Hamiltonian Monte Carlo or Metropolis-adjusted Langevin algorithms, the method operates on an ensemble and can be implemented in a derivative-free manner. Compared to ensemble Kalman and consensus-based approaches, ALDI is formulated as a stochastic dynamics that admits the target distribution as an invariant measure. In contrast, many ensemble-based optimization and sampling methods, including standard consensus-based approaches and the derivative-free version of ALDI, do not in general possess such an exact invariance property. An extensive systematic quantitative comparison with state-of-the-art rare event sampling methods is beyond the scope of the present work. However, a first representative numerical comparison with the consensus-based method of \cite{doi:10.1137/23M1565966} is provided in Subsection~\ref{SEC:cbree_comparison}; a more comprehensive benchmarking study across multiple methods and problem classes remains an interesting direction for future work.

\subsection{A convex limit-state function}
\label{SUBSEC:convex}
We consider the limit-state function
\begin{equation}
    \label{eq:convex_example}
    G(x_1,x_2)
    = 0.1 (x_1 - x_2)^2 - \frac{1}{\sqrt{2}}(x_1 + x_2) + 2.5,
\end{equation}
defines a convex failure domain.  As shown in \cite{Pappaioannou16}, the
distribution of $(x_1,x_2)$ conditioned on $\{G \le 0\}$ is unimodal {when the prior distribution is standard normal}, making this
a convenient test case for illustrating the behaviour of ALDI. In the following numerical experiments, we fix the smoothing parameter $\delta=0.001$, and the prior distribution $\mu_0\sim \mathcal{N}\left((0,0), Id \right)$. We solve ALDI numerically using the Euler-Maruyama method with a time step $\Delta \tau=0.001$ on an integration time window $\tau=40$. We used two Gaussian components in the importance sampling step.

\paragraph{Numerical results}
We illustrate the behaviour of both gradient-based (GB) and gradient-free (GF) ALDI on the convex limit--state function
\eqref{eq:convex_example}. Figure~\ref{fig:convex_scatter} displays the final ALDI ensembles for three
ensemble sizes $J\in\{100,500,1000\}$ and two values of the observational noise
parameter $R\in\{0.1,0.01\}$.
Smaller values of $R$ lead to a stronger concentration of particles near the
failure boundary $\{G=0\}$, while larger values of $R$ result in a more diffuse
distribution dominated by the prior.
Both GB and GF variants of ALDI exhibit the same qualitative behaviour, with the GF ALDI final ensemble having a more spread cloud, primarily visible for small ensembles. The reference solution is $P_f=4.21\times 10^{-3}$, cf.~\cite{Pappaioannou16} and verified by Monte Carlo.

\smallskip

\begin{figure}[h!]
    \centering
    \begin{overpic}[width=\linewidth]{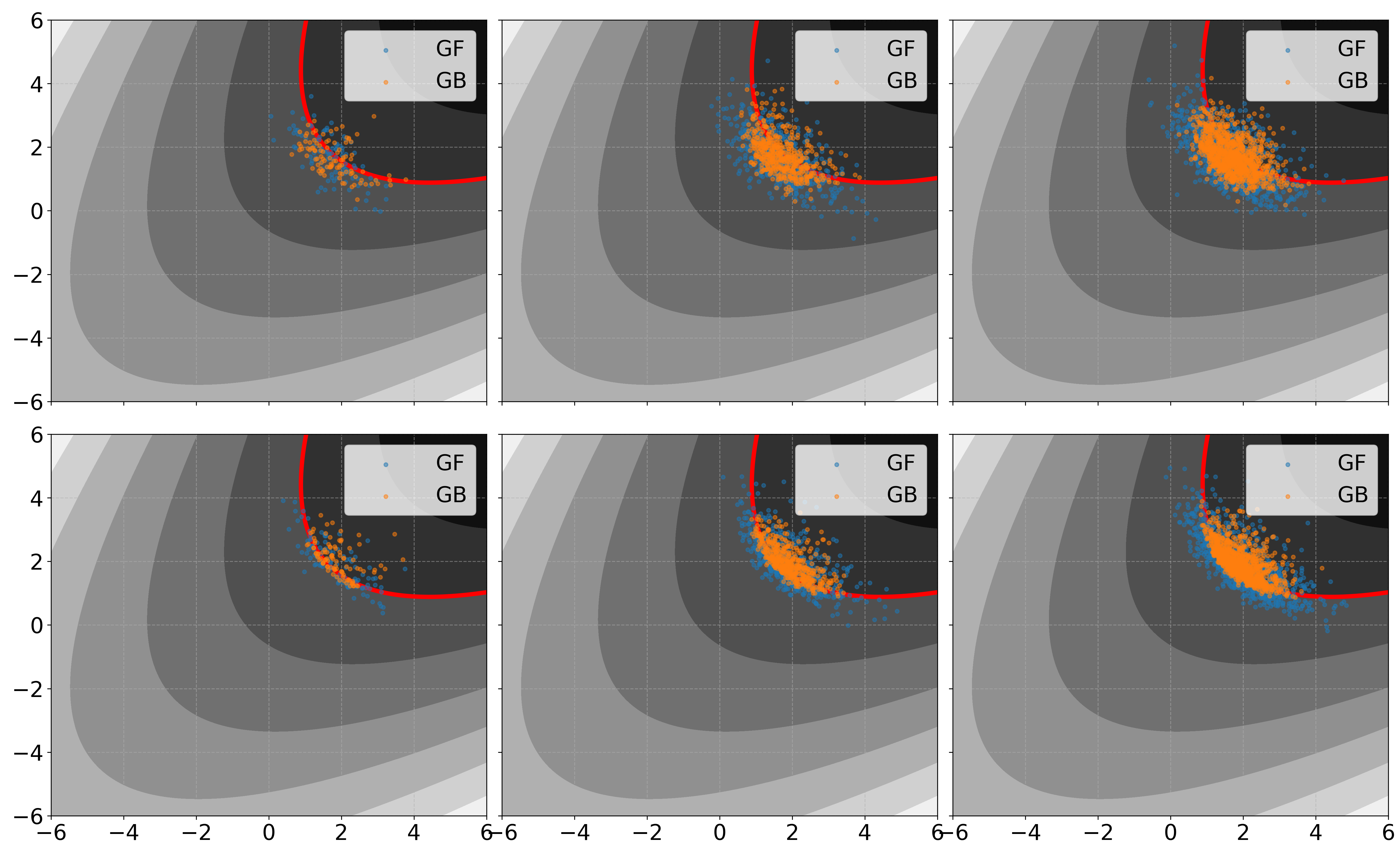}
           \put(15,61){\footnotesize$J=100$}
           \put(48,61){\footnotesize$J=500$}
           \put(80,61){\footnotesize$J=1000$}
           \put(-8,45){\footnotesize $R=0.1$}
           \put(-8,15){\footnotesize $R=0.01$}
    \end{overpic}
    \caption{Final ALDI ensembles for the convex example.
    Columns correspond to ensemble sizes $J=100,1000,10000$ from left to right, and rows correspond to $R=0.1,0.001$ from top to bottom.}
    In blue and yellow, the final ensembles for GF and GB ALDI are presented, respectively. The red curve denotes the failure boundary $\{G=0\}$.
    \label{fig:convex_scatter}
\end{figure}

The effect of the ensemble size $J$ on Alg.~\ref{alg:aldi_product} is shown in
Figure~\ref{fig:size_and_failed}.  
Although the identity $P_f = Z^*_{\delta,R}\,\mathbb P_{\rho^*_{\delta,R}}(X\in F)$
is exact and motivates Alg.~\ref{alg:aldi_product}, estimating the normalising constant
$Z^*_{\delta,R}$ directly from stationary ALDI samples (with $M=J$) via 
\[
\widehat Z^*_{\delta,R}
=\left(\frac1 J\sum_{m=1}^J \exp\!\Big(\tfrac{1}{2R}\tilde G_\delta(X_m)^2\Big)\right)^{-1}
\]
is typically numerically unstable. The
term $\exp(\tilde G_\delta^2(X_m)/(2R))$ can be extremely large away from the failure
set for $R$ small, so that rare excursions of the ALDI chain dominate the empirical average and
lead to very high variance. In addition, any deviation from stationarity introduces
bias, since the inverse-moment identity holds only under $\rho^*_{\delta,R}$. Indeed, if the empirical ALDI distribution differs from $\rho^*_{\delta,R}$,
the inverse-moment estimator converges to an expectation under the transient
distribution rather than $(Z_{\delta,R}^{*})^{-1}$. In addition, time discretization introduces a bias in the invariant measure of
the numerical ALDI scheme, which further degrades the accuracy of the normalisation
constant estimator.
For these reasons, we focus in the following on Alg.~\ref{alg:aldi_mixture_is}, which avoids estimating
$Z^*_{\delta,R}$ altogether by working with an explicit fitted proposal.

\medskip

\begin{figure}[h!]
    \centering
    \begin{subfigure}[t]{0.45\textwidth}
        \centering
        \begin{overpic}[width=\textwidth]{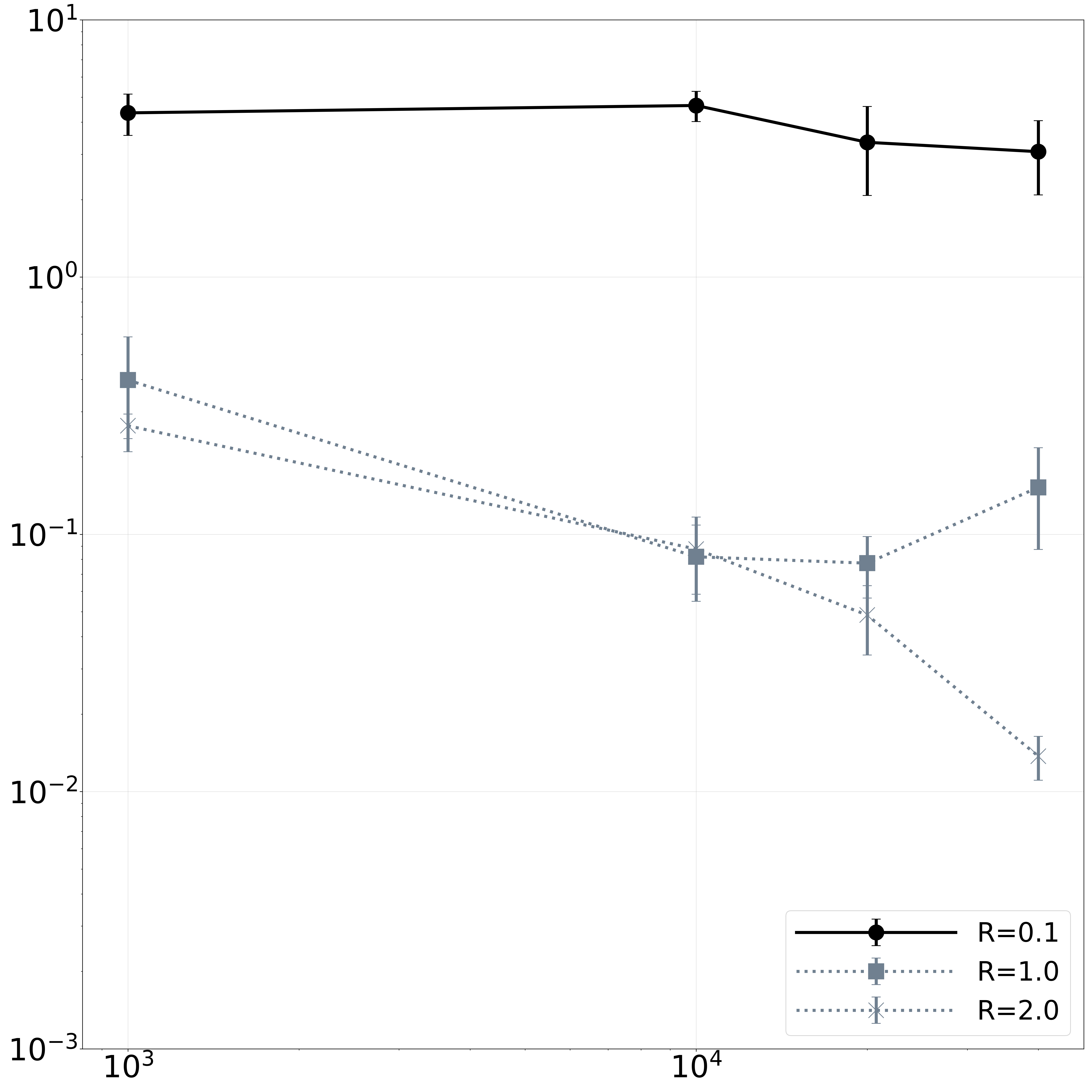}
            \put(50,-2){\footnotesize $J$}
            \put(-10,45){\footnotesize \begin{turn}{90} $\frac{\lvert \widehat P_f- P_{\mathrm{ref}}\rvert}{P_{\mathrm{ref}}}$ \end{turn}}
        \end{overpic}
        \caption{}
    \end{subfigure}
    \hfill
    \begin{subfigure}[t]{0.45\textwidth}
        \centering
        \begin{overpic}[width=\textwidth]{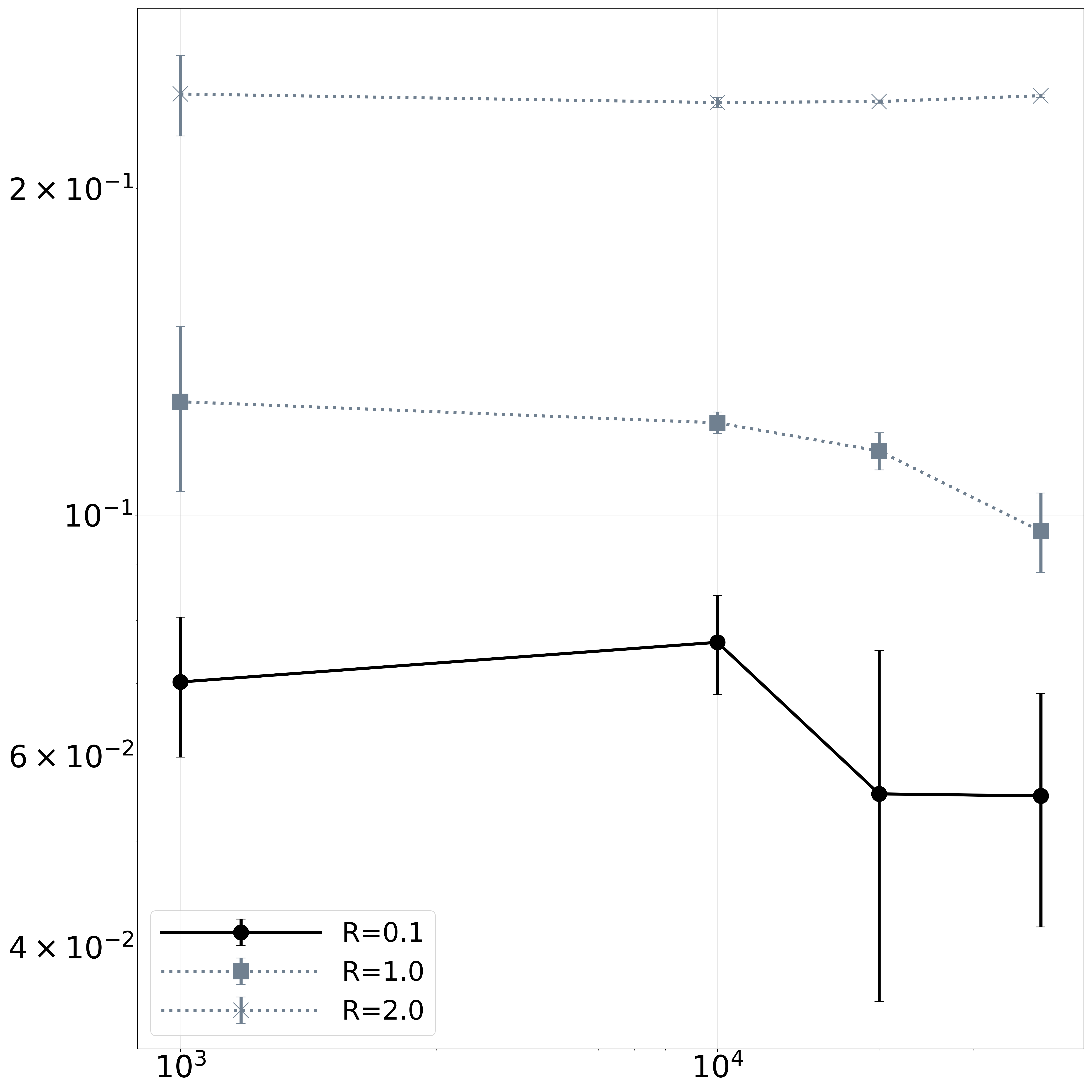}
            \put(50,-2){\footnotesize $J$}
            \put(-7,45){\footnotesize \begin{turn}{90} $\widehat{Z}_{\delta,R}^*$ \end{turn}}
        \end{overpic}
        \caption{}
    \end{subfigure}
    \caption{Dependence of the GB ALDI estimator on the ensemble size $J$ for Alg.~\ref{alg:aldi_product}. In (a), the relative error is shown as a function of $J$, where we observe that better results are obtained for larger values of $R$, namely $R=1,2$. In (b), the estimated normalisation constant $\widehat{Z}^*_{\delta,R}$ is shown as a function 
of the ensemble size $J$ across independent runs. The direct estimation of 
$Z^*_{\delta,R}$ from ALDI samples gets unstable for smaller values of $R$.}
    \label{fig:size_and_failed}
\end{figure}

Next, we investigate the dependence on the ensemble size $J$. The number of particles in the final ensemble located inside the failure domain increases
with $J$, indicating improved exploration of the region of the rare event and reduced
weight degeneracy in the subsequent importance-sampling step, see Figure~\ref{fig:gm_components} (a).
We further examine the influence of the number of Gaussian mixture components
used in the importance-sampling proposal.
Figure~\ref{fig:gm_components} (b) shows that already one or two Gaussian components suffice to achieve accurate estimates, with little sensitivity to further increasing the number of components.

\begin{figure}[h!]
    \centering
    \begin{subfigure}[t]{0.45\textwidth}
        \centering
         \centering
        \begin{overpic}[width=\textwidth]{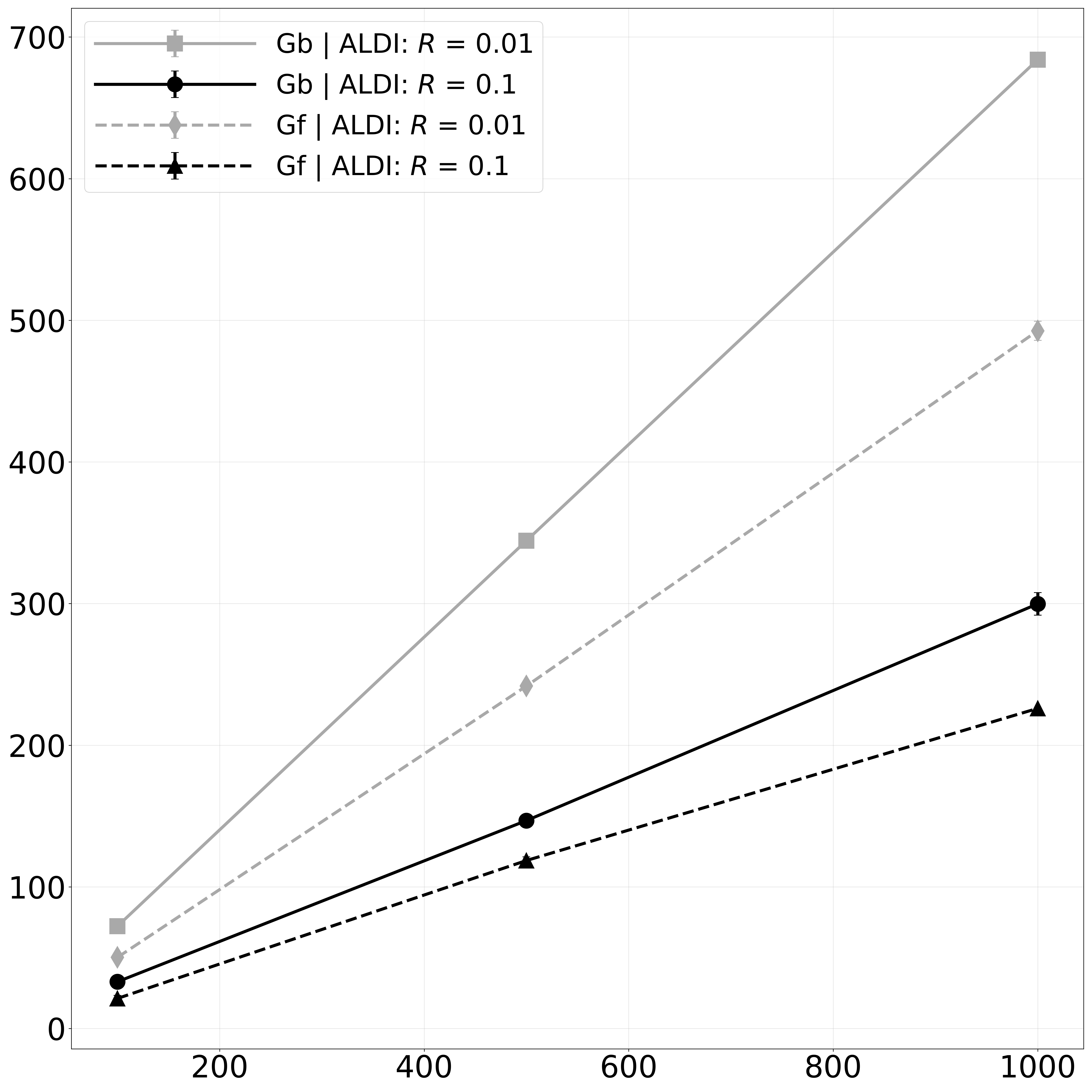}
        \put(50,-2){\footnotesize$J$}
        \put(-6,35){\footnotesize \begin{turn}{90} $\#$ particles in $F$\end{turn}}
        \end{overpic}
        \caption{}
    \end{subfigure}
    \hfill
    \begin{subfigure}[t]{0.45\textwidth}
        \centering
        \begin{overpic}[width=\textwidth]{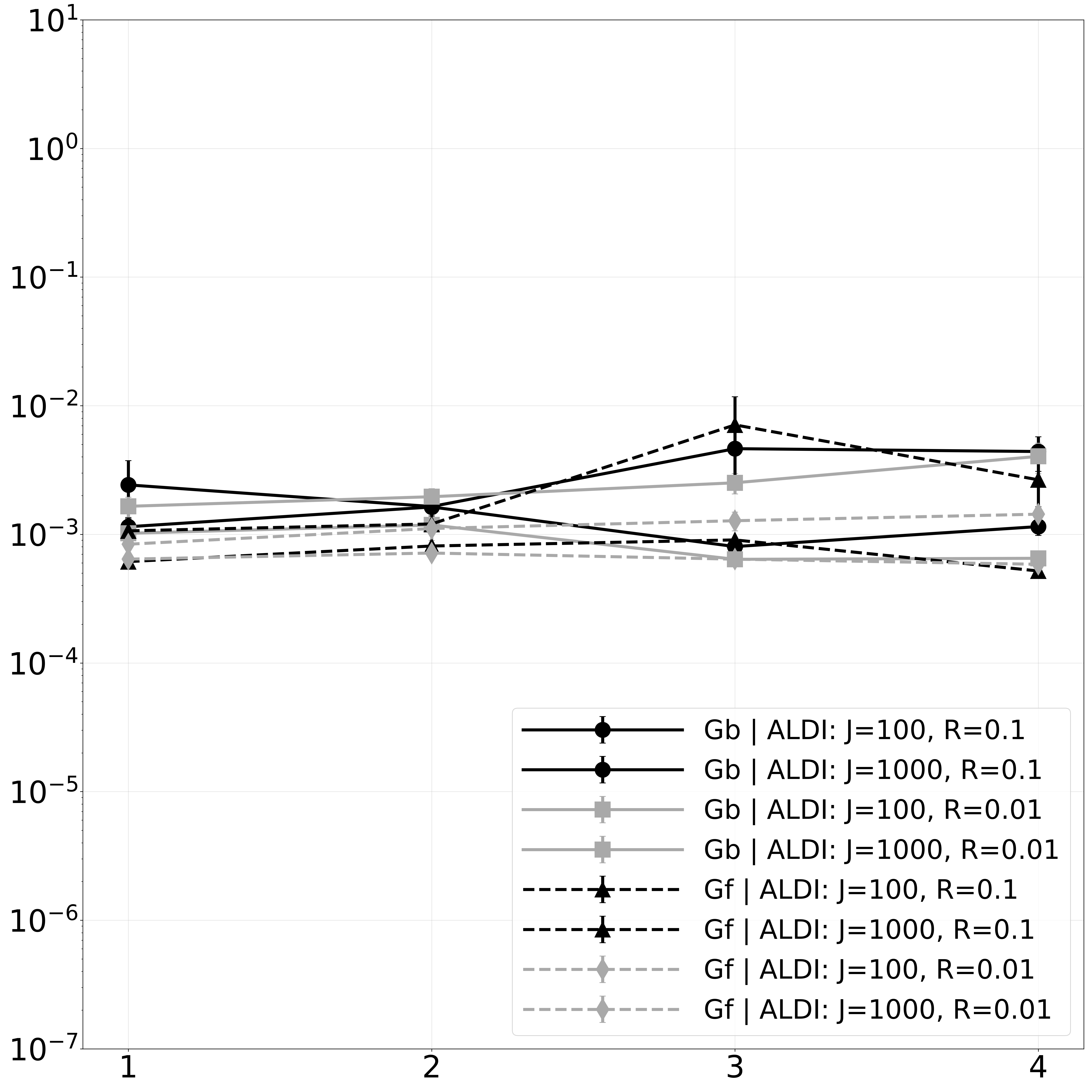}
        \put(13,-2){\footnotesize $\#$ Gaussian mixture components}
        \put(-12,40){\footnotesize \begin{turn}{90} $\frac{ \lvert \widehat P_f^{\mathrm{IS}}- P_{\mathrm{ref}}\rvert}{P_{\mathrm{ref}}}$
        \end{turn}}
    \end{overpic}
        \caption{}
    \end{subfigure}
    \caption{In (a), the number of particles lying inside of the failure domain as a function of $J$ is shown. In (b), we show the relative error {of the estimator $\hat{P}_f^{\mathrm{IS}}$} as a function of the number of Gaussian mixture
    components used in the importance-sampling proposal Alg.~\ref{alg:aldi_mixture_is} with $10^3$ samples.}
    \label{fig:gm_components}
\end{figure}

We finally assess the accuracy of the subsequent importance sampling step based
on the ALDI-generated proposal.
Figure~\ref{fig:is_error_nonrare} reports the relative error of the IS estimator as a
function of the total number of importance samples.
The IS error decays at the expected Monte Carlo rate across all ensemble sizes and already from small numbers of samples, confirming that the fitted Gaussian mixture constitutes an accurate proposal distribution.
The IS errors for different ALDI ensemble sizes are nearly indistinguishable, indicating that even a modest ensemble of 
$J=100$ particles suffices to produce a well-adapted proposal distribution.
\begin{figure}[h!]
    \centering
    \begin{subfigure}[t]{.44\textwidth}
    \centering
    \begin{overpic}[width=\textwidth]{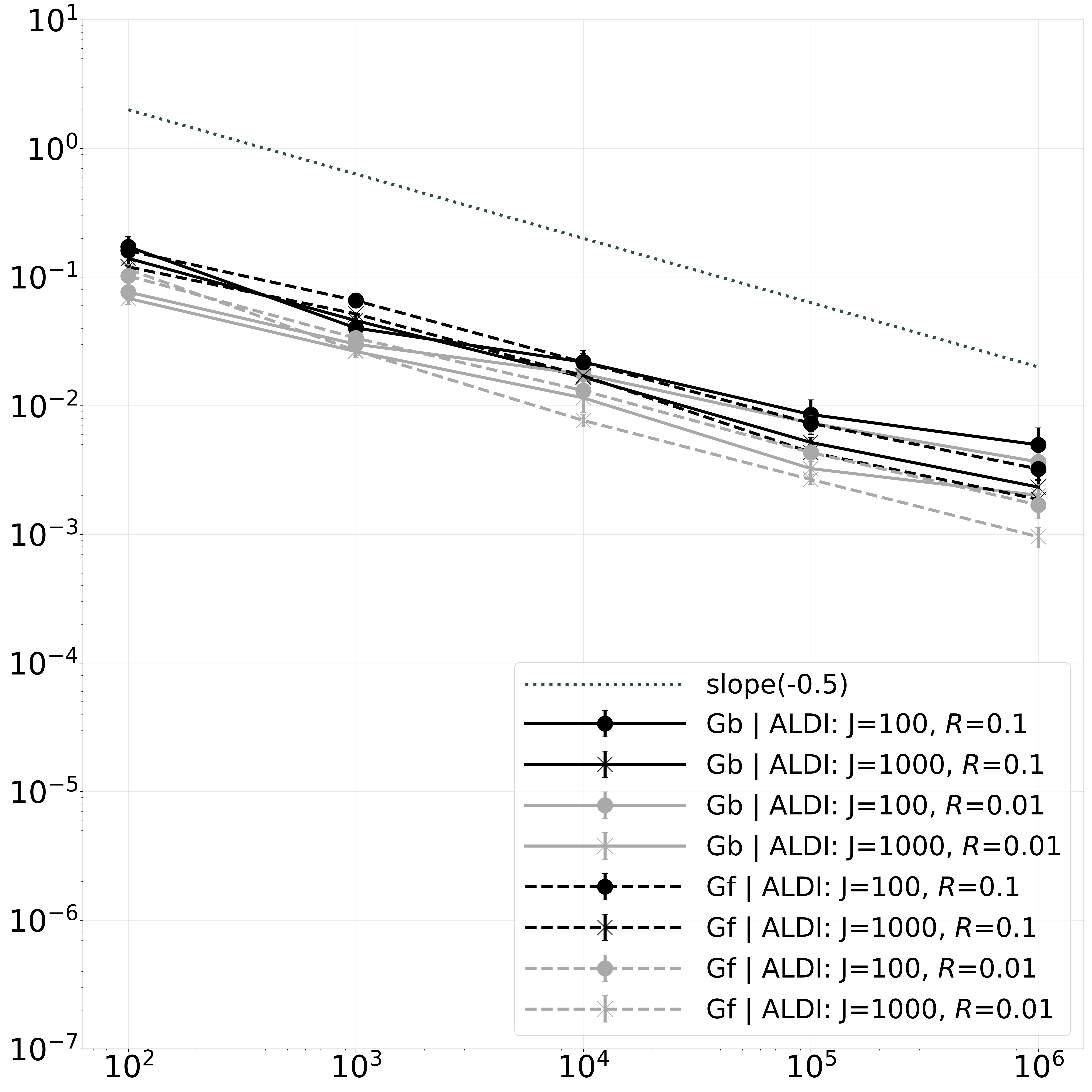}
        \put(-12,35){\footnotesize \begin{turn}{90}
        $\frac{\lvert \widehat P_f^{\mathrm{IS}}- P_{\mathrm{ref}}\rvert}{P_{\mathrm{ref}}}$
        \end{turn}}
        \put(40,-2){\footnotesize $\#$ IS samples}
    \end{overpic}
    \caption{}
    \end{subfigure}
    \hfill
    \begin{subfigure}[t]{.44\textwidth}
    \centering
    \begin{overpic}[width=\textwidth]{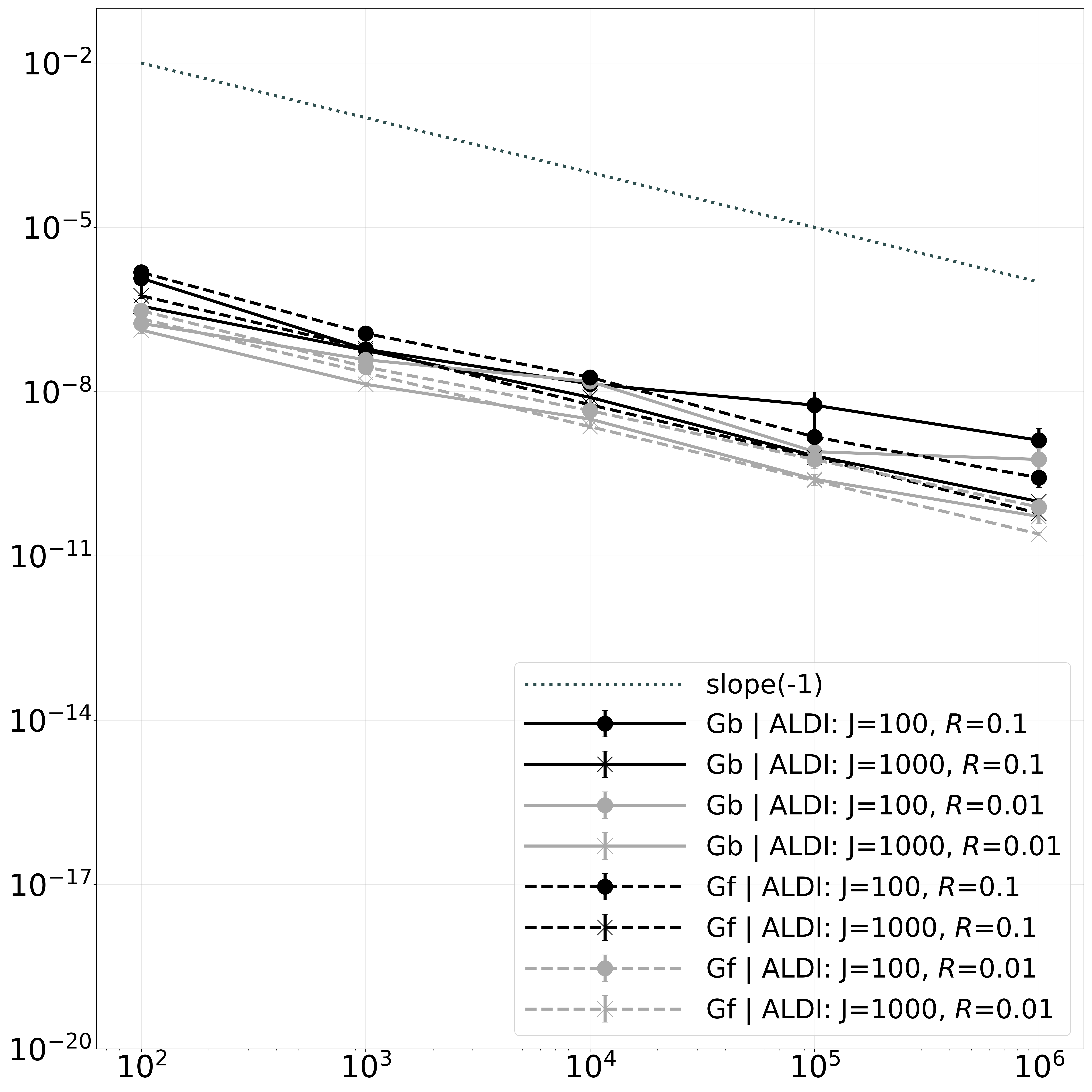}
        \put(-12,35){\footnotesize \begin{turn}{90}        $Var\left(\widehat{P}^{\mathrm{IS}}_{f}\right)$
        \end{turn}}
        \put(40,-2){\footnotesize $\#$ IS samples}
    \end{overpic}
    \caption{}
    \end{subfigure}
    \caption{The relative error and the variance of the IS estimator in Alg.~\ref{alg:aldi_mixture_is} as a function of
    the number of IS samples for $J=100,1000$ and $R=0.1,0.001$ are portrayed, respectively.    
    As a reference, we include the numerical results from Monte Carlo simulations, along with  straight lines corresponding to their theoretical estimation. We confirm that that both quantities follow closely the analytical prediction.}
    \label{fig:is_error_nonrare}
\end{figure}


\smallskip

\paragraph{A rarer-event regime beyond Monte Carlo}
We conclude with a more challenging variant of the convex example in which the
failure event is significantly rarer.
In this regime, prior-based IS becomes ineffective: within feasible
computational budgets, no failure samples are observed.
 In the following, we fixed the prior distribution as Gaussian with mean $(-2,-2)$ and covariance matrix $0.8 Id$, and solved ALDI using the Euler-Maruyama method with time step $\Delta\tau=0.0005$ up to time $\tau=20$. The smoothing parameter $\delta=0.001$ was taken and $2$ Gaussian components in the importance sampling step. The reference probability of failure in this case could not be obtained solely through Monte Carlo experiments, as no realisations entered the failure domain even after \(10^8\) iterations. Instead, we take as reference value the average of five ALDI runs (ensemble size $103$) with $J=10^8$ importance sampling particles. This approach yielded $P_{\mathrm{ref}}=7.1425\times 10^{-10}$.\\

Figure~\ref{fig:rare_scatter} shows the final ALDI ensembles for three ensemble
sizes and two noise levels.
Despite the increased rarity of the event, ALDI continues to concentrate
particles close to the failure boundary, thereby producing informative proposal
distributions for importance sampling.
Both GB and GF variants exhibit stable behaviour, with smaller values of $R$ leading to a sharper concentration of particles near the failure boundary and a more accurate approximation of the rare-event region, while the ensemble size $J$ plays a lesser role here.

\begin{figure}[h!]
    \centering
    \begin{overpic}[width=\linewidth]{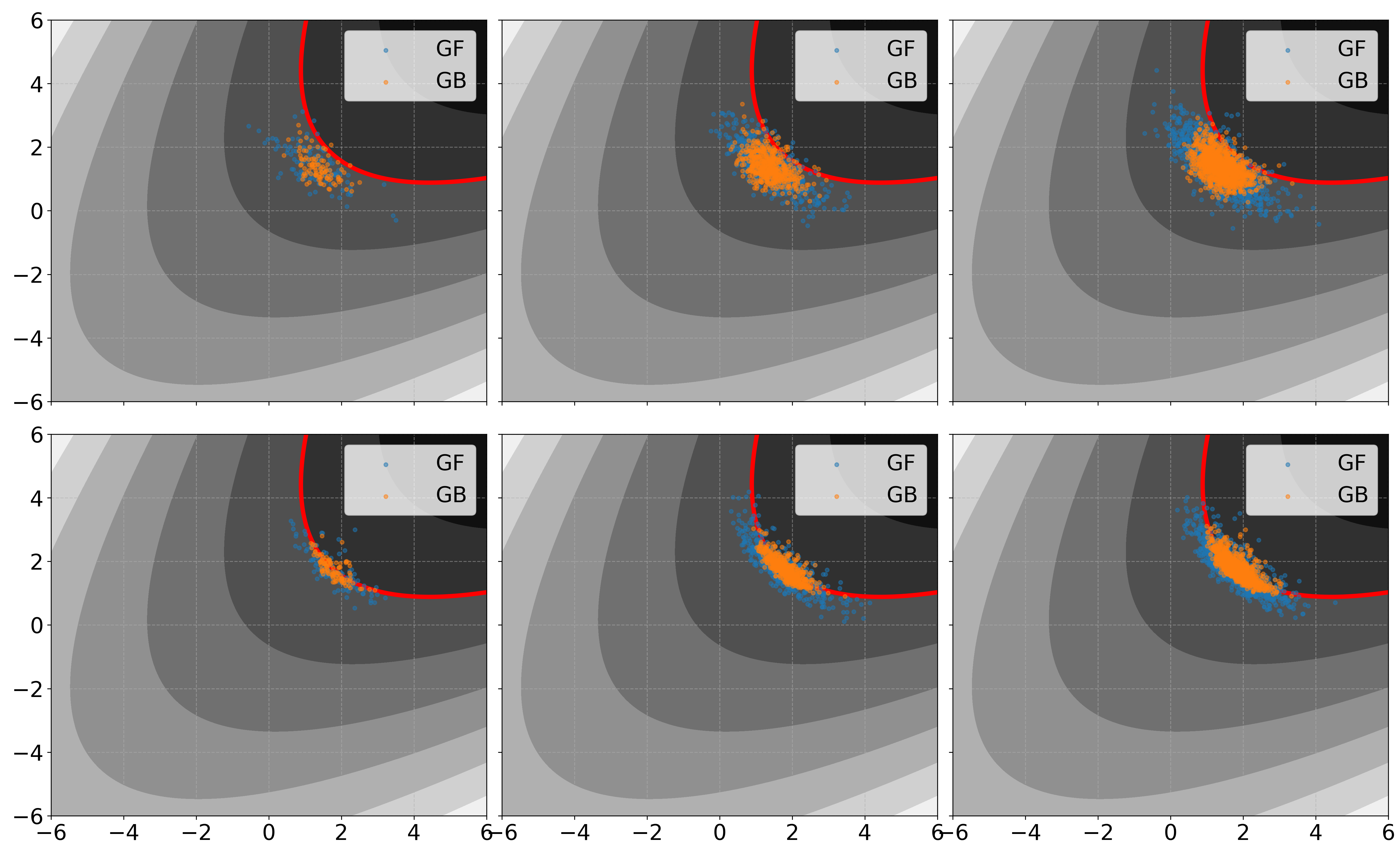}
          \put(15,61){\footnotesize$J=100$}
           \put(48,61){\footnotesize$J=500$}
           \put(80,61){\footnotesize$J=1000$}
           \put(-8,45){\footnotesize $R=0.1$}
           \put(-8,15){\footnotesize $R=0.01$}
    \end{overpic}
    \caption{Final ALDI ensembles in a rarer-event regime. Similarly as in Figure~\ref{fig:convex_scatter}, we portray the cases for $J=100,500,1000$ (left to right) and $R=0.1,0.01$ (top to bottom). The blue ensemble corresponds to the GF ALDI whereas the yellow on corresponds to the GB ALDI. The red curve denotes the failure boundary $\{G=0\}$.}
    \label{fig:rare_scatter}
\end{figure}

We next assess the performance of the subsequent importance sampling step.
Figure~\ref{fig:is_error_rare} reports the estimated failure probability and the
corresponding variance as functions of the number of importance samples.
In contrast to prior-based IS, the ALDI-based importance sampling estimator remains stable in the rare-event regime. Smaller values of $R$ lead to a systematic reduction in the estimated variance, indicating that a sharper concentration of the proposal near the failure boundary directly translates into more efficient importance sampling. In contrast, sampling from the prior fails to capture the rare-event region altogether and therefore does not provide a meaningful estimator.

\begin{figure}[h!]
    \centering
    \begin{subfigure}[t]{.45\textwidth}
    \centering
    \begin{overpic}[width=\textwidth]{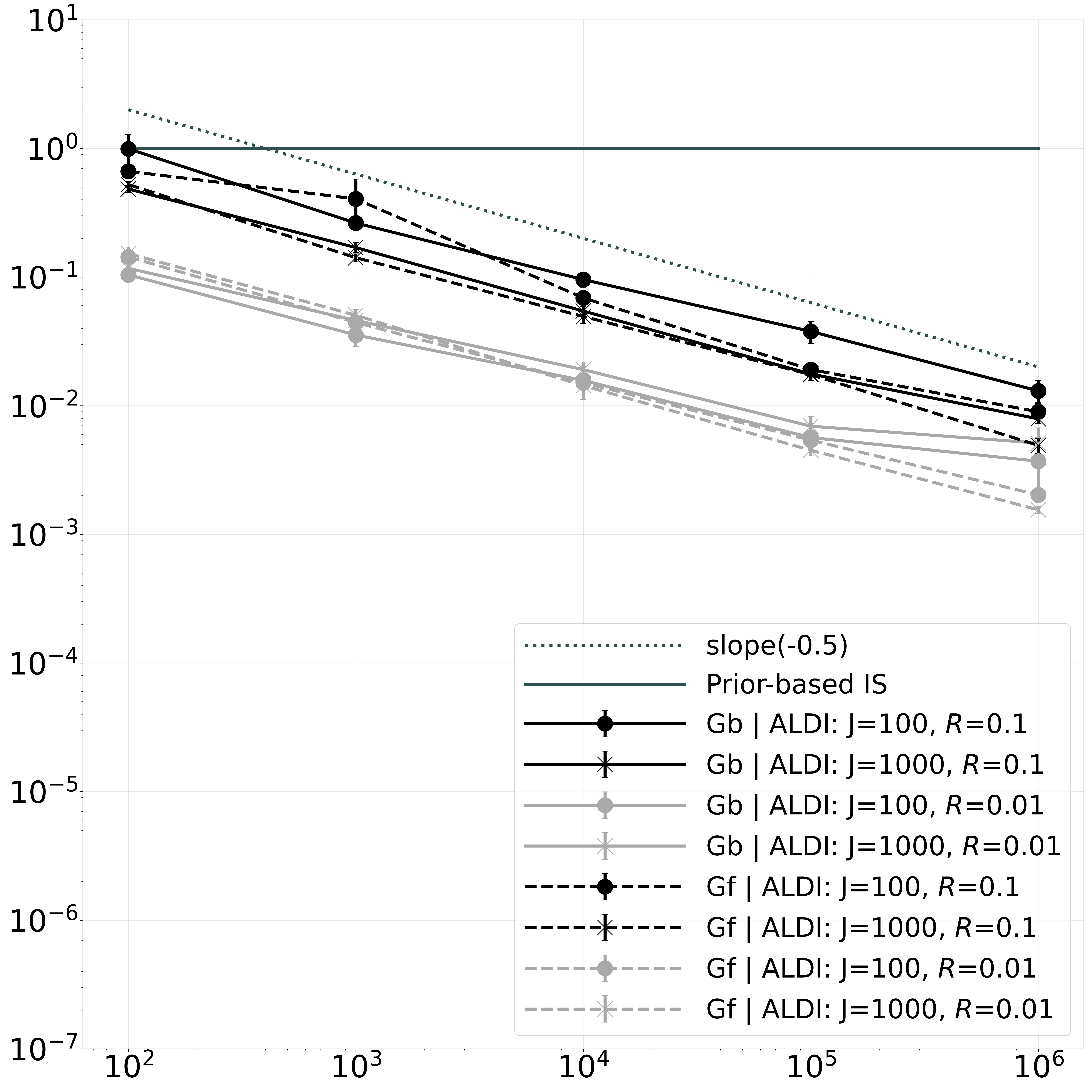}
        \put(-12,35){\footnotesize \begin{turn}{90}
        $\frac{\lvert \widehat P_f^{\mathrm{IS}}- P_{\mathrm{ref}}\rvert}{P_{\mathrm{ref}}}$
        \end{turn}}
        \put(40,-2){\footnotesize $\#$ IS samples}
    \end{overpic}
    \caption{}
    \end{subfigure}
    \hfill
    \begin{subfigure}[t]{.45\textwidth}
    \centering
    \begin{overpic}[width=\textwidth]{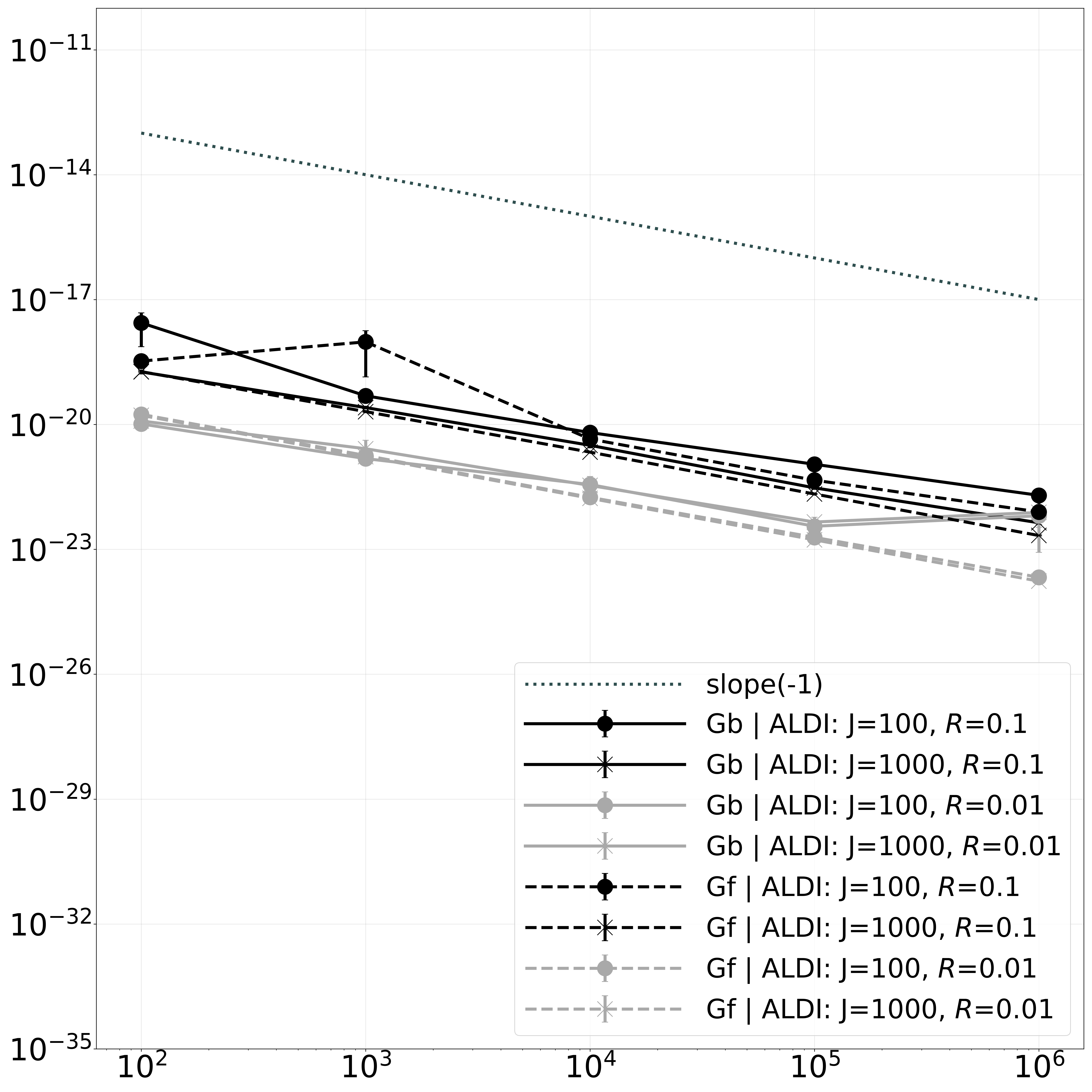}
        \put(-12,35){\footnotesize \begin{turn}{90}        $Var\left(\widehat{P}^{\mathrm{IS}}_{f}\right)$
        \end{turn}}
        \put(40,-2){\footnotesize $\#$ IS samples}
    \end{overpic}
    \caption{}
    \end{subfigure}
    \caption{The relative error and the variance of the IS estimator in Alg.~\ref{alg:aldi_mixture_is} as a function of
    the number of IS samples for $J=100,1000$ and $R=0.1,0.001$ are portrayed, respectively, in a rarer-event regime.    
    Here, the prior-based IS does not show a sensible result, as the experiments do not reach the failure domain. We include the Monte Carlo theoretical references}
    \label{fig:is_error_rare}
\end{figure}

\subsubsection{Comparison with a consensus-based method}
\label{SEC:cbree_comparison}

We now compare the performance of ALDI with the Consensus-Based Rare Event Estimation (CBREE) method of \cite{doi:10.1137/23M1565966} on the convex rarer-event regime described above. CBREE constructs a sequence of approximations to the failure probability via an ensemble of interacting particles governed by a McKean--Vlasov (SDE) derived from consensus-based sampling; the saturation parameter (coefficient of variation) $c_{\mathrm{var}}$ controls how aggressively the particle distribution is pushed toward the failure domain. For the comparison, the ALDI time-step \(\Delta \tau = 0.01\) and total number of forward evaluations were limited to \(10^4\) steps. For the CBREE method, we implemented the code in \cite{AlthausCode23}, with \(10^4\) forward evaluations and (\(c_{\mathrm{var}}\)) fixed. The smoothing parameters, the inverse temperature \(\beta\) and the step size $h$ vary adaptively. For more details on the choice of parameters, we refer to \cite{doi:10.1137/23M1565966}.

Figure~\ref{fig:is_error_comparison} shows the relative IS error (a) and estimated variance (b) as functions of the number of IS samples. Both ALDI (GB and GF variants, $J\in\{100,1000\}$, $R\in\{0.1,0.01\}$) and CBREE ($J\in\{100,1000\}$, $c_{\mathrm{var}}\in\{1.5,2.5\}$) achieve the expected Monte Carlo rates: the relative error decays at rate $N^{-1/2}$ and the variance at rate $N^{-1}$, confirming that both methods provide effective proposal distributions for importance sampling. ALDI with the smaller noise parameter $R=0.01$ yields accuracy comparable to the best CBREE configurations, while the influence of ensemble size $J$ is modest for both methods. These results indicate that ALDI is competitive with the consensus-based approach on this benchmark, despite resting on a fundamentally different algorithmic principle.
\begin{figure}[h!]
    \centering
    \begin{subfigure}[t]{.45\textwidth}
    \centering
    \begin{overpic}[width=\textwidth]{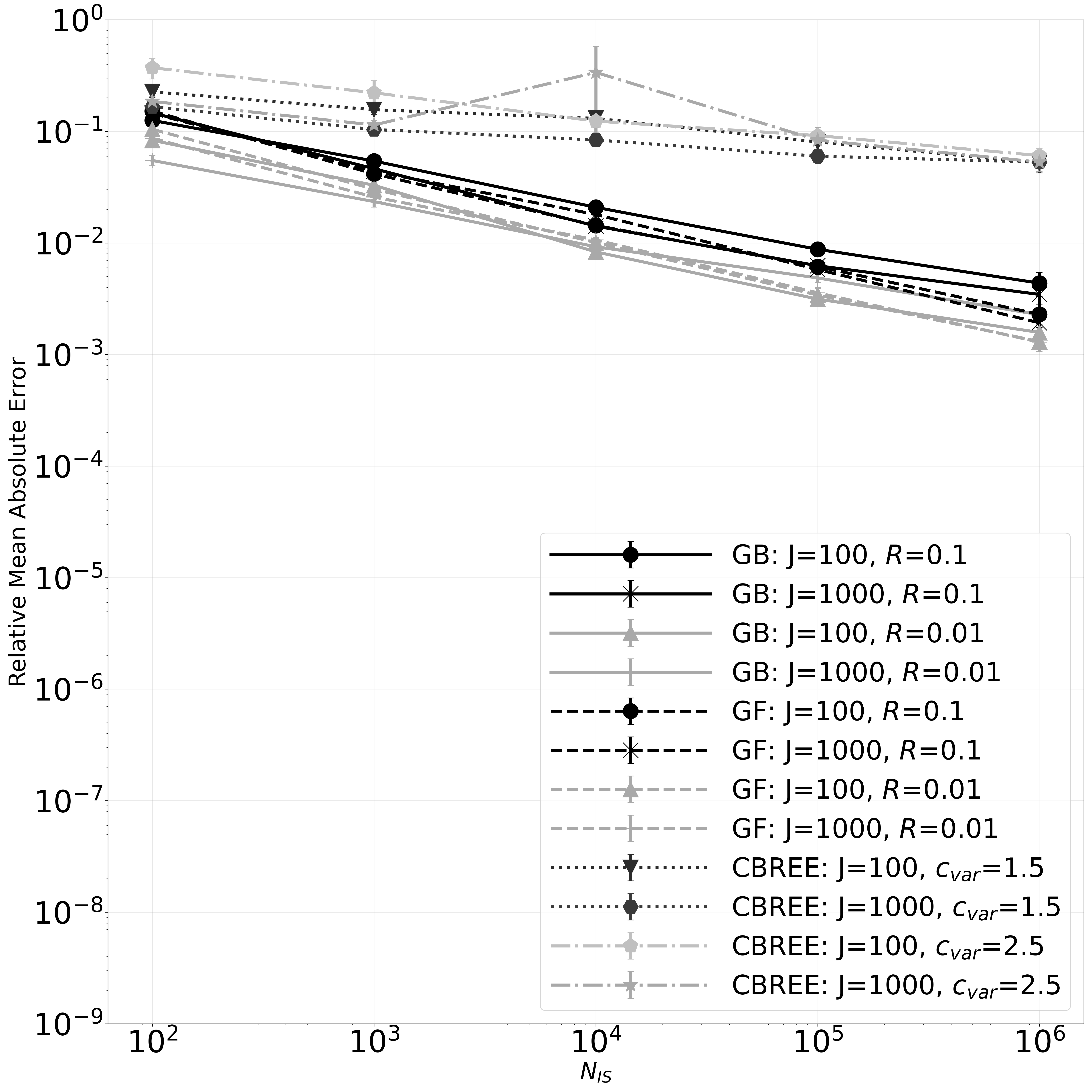}
        \put(-12,35){\footnotesize \begin{turn}{90}
        $\frac{\lvert \widehat P_f^{\mathrm{IS}}- P_{\mathrm{ref}}\rvert}{P_{\mathrm{ref}}}$
        \end{turn}}
        \put(40,-2){\footnotesize $\#$ IS samples}
    \end{overpic}
    \caption{}
    \end{subfigure}
    \hfill
    \begin{subfigure}[t]{.45\textwidth}
    \centering
    \begin{overpic}[width=\textwidth]{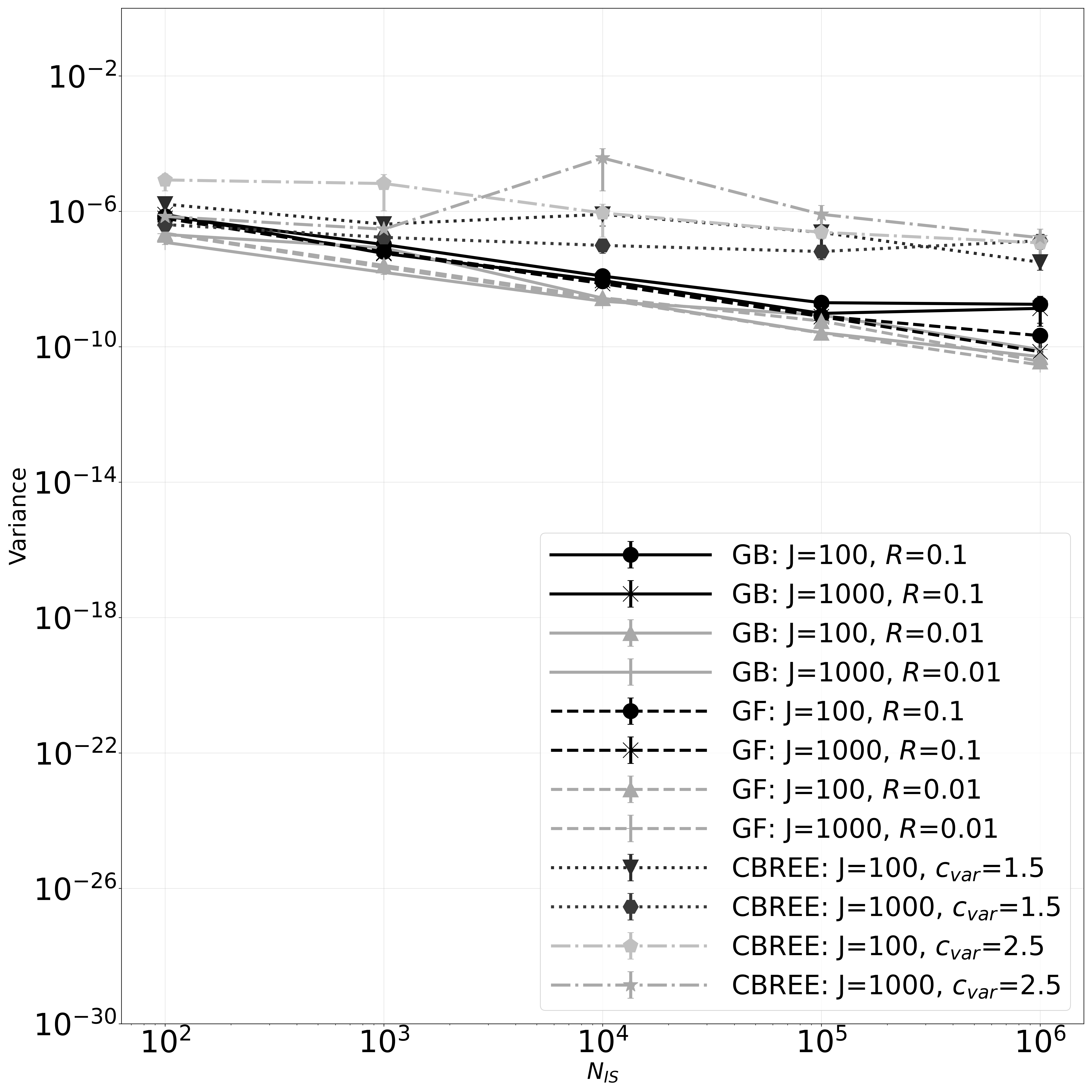}
        \put(-12,35){\footnotesize \begin{turn}{90}        $Var\left(\widehat{P}^{\mathrm{IS}}_f\right)$
        \end{turn}}
        \put(40,-2){\footnotesize $\#$ IS samples}
    \end{overpic}
    \caption{}
    \end{subfigure}
    \caption{We compare the relative error and the variance for both ALDI methods with a consensus-based estimator. In both cases, the ensemble sizes considered were $J=100,1000$. In ALDI, the values  $R=0.1,0.01$ were taken, whereas for the consensus-based method $c_{\mathrm{var}}=1.5,2.5$ were employed.}
    \label{fig:is_error_comparison}
\end{figure}
It is worth noting that the ALDI implementation has not yet been optimised with respect to computational efficiency---in contrast to the highly tuned parameter adaptation in CBREE \cite{doi:10.1137/23M1565966}. The fact that comparable accuracy is achieved nonetheless suggests that the ALDI framework for rare event estimation is a promising direction for further investigation.

\subsection{A hyperbolic saddle}
\label{SUBSEC:saddle}
We next consider a simple linear example that captures some of the dynamical
features underlying the point-vortex model discussed later in
Section~\ref{SUBSEC:vortex}.  Let $(x(t),y(t))$ solve
\begin{equation}
\label{eq:saddle_eq}
\begin{aligned}
    \dot{x} &= -\lambda x,\\
    \dot{y} &= \mu y,
\end{aligned}
\end{equation}
with $\lambda,\mu>0$.  The origin is a saddle equilibrium with one-dimensional
stable and unstable manifolds given by the horizontal and vertical axes,
respectively.  The explicit solution for initial data $(x_0,y_0)$ is
\begin{equation}
\label{eq:saddle_sols}
    x(t)=x_0 e^{-\lambda t}, 
    \qquad 
    y(t)=y_0 e^{\mu t}.
\end{equation}

Unless $y_0=0$, the trajectory escapes exponentially fast along the unstable
direction, given by the vertical axis.  We consider the rare event that a trajectory remains close to the
origin over a finite time window $[0,T]$.  This is quantified by the observable
\begin{equation}
    \label{eq:obs1_saddle}
    A(x_0,y_0)
    := \frac{1}{T} \int_0^T \bigl(x(s)^2 + y(s)^2 \bigr)\, ds.
\end{equation}
A short computation gives the closed-form expression
\[
A(x_0,y_0)
= \frac{1}{T}
\left[
  \frac{1-e^{-2\lambda T}}{2\lambda}\, x_0^2
  +
  \frac{e^{2\mu T}-1}{2\mu}\, y_0^2
\right],
\]
so that the rare-event set $\{A<r\}$, for $r\ll 1$, is the solid ellipse
\begin{equation}
\label{eq:failure_saddle_ellipse}
\mathcal{E}
=
\left\{
(x_0,y_0)\in\mathbb{R}^2:\ 
\left(\frac{1-e^{-2\lambda T}}{2\lambda r T}\right) x_0^2
+
\left(\frac{e^{2\mu T}-1}{2\mu r T}\right) y_0^2 
< 1
\right\}.
\end{equation}

We define the limit-state function
\begin{equation}
    \label{eq:LSF_saddle}
    G(x_0,y_0) = A(x_0,y_0) - r.
\end{equation}
For Gaussian prior samples $(x_0,y_0)\sim \mathcal{N}(m,\sigma^2 I)$, the
probability of failure is
\[
P_f
=
\iint_{\mathcal{E}} 
\rho_0(x,y)\,dx\,dy,
\qquad
\rho_0(x,y)
=
\frac{1}{2\pi\sigma^2}
\exp\!\left(
-\frac{(x-m_x)^2 + (y-m_y)^2}{2\sigma^2}
\right).
\]

This example is particularly convenient because the failure set is an ellipse
and $P_f$ can again be computed numerically to high accuracy.  This allows us to assess the performance of ALDI in a dynamical setting featuring unstable behaviour, where the rare event corresponds to trajectories remaining close to the saddle point.

The corresponding potential is
\begin{equation}
    \label{eq:potential_saddle}
    \Phi(x_0,y_0)
    =
    -\log \rho_0(x_0,y_0)
    +
    \frac{1}{2R}\,\tilde{G}(x_0,y_0)^2,
\end{equation}
where $\tilde G=\max\{0,G\}$ and $G(x_0,y_0)=A(x_0,y_0)-r$ with
\begin{equation}\label{eq:A_saddle}
A(x_0,y_0)
= \frac{1}{T}\Bigg(
   \frac{1-e^{-2\lambda T}}{2\lambda}x_0^2
 + \frac{e^{2\mu T}-1}{2\mu}y_0^2
  \Bigg).
\end{equation}
The gradient of $G$ is computed explicitly as
\begin{equation}\label{eq:gradG_saddle}
\nabla G(x_0,y_0)
 = \frac{1-e^{-2\lambda T}}{\lambda T}\,x_0\,\mathbf e_x
 + \frac{e^{2\mu T}-1}{\mu T}\,y_0\,\mathbf e_y.
\end{equation}
Hence,
\begin{equation}\label{eq:gradPhi_saddle}
\nabla\Phi(x_0,y_0)
 = \frac{1}{\sigma^2}
   \begin{pmatrix}x_0-m_x\\ y_0-m_y\end{pmatrix}
 + \frac{\tilde G(x_0,y_0)}{R}\,
   \nabla G(x_0,y_0).
\end{equation}

\paragraph{Numerical results.}
We now illustrate the behaviour of ALDI for the saddle-type example introduced
above. In the following experiments, we fix $\lambda=\mu=1$, the time horizon $T=$ 1, the tolerance $r=0.5$, and the prior distribution to be normally distributed with mean \((-1,-1)\) and covariance matrix $0.5 Id$. Additionally, we solve ALDI numerically using the Euler-Maruyama method, with time step $\Delta \tau=\frac{0.001}{4}$ up to time $\tau=5$. In the importance sampling step, the number of Gaussian components used was 1. The reference probability, $P_{\mathrm{ref}} = 0.0676$, was computed by averaging the results of five independent Monte Carlo runs, each utilizing $10^8$ samples.

Since the failure set $\{A<r\}$ is an ellipse and its probability under the
Gaussian prior can be computed accurately, this setting provides a controlled
benchmark for assessing the estimator $\widehat P_f$ or $\widehat P_f^{\mathrm{IS}}$ in a dynamical situation
characterised by instability along the vertical axis. All failure probability estimates in this section are computed using the mixture-based importance sampling estimator (Alg. \ref{alg:aldi_mixture_is}).

Figure~\ref{fig:saddle_scatter} shows representative ALDI ensembles for different ensemble
sizes $J\in\{100,500,1000\}$ and noise levels $R\in\{1,0.1,0.01\}$. As $R$ decreases,
the invariant measure concentrates increasingly near the failure ellipse, reflecting
the stronger selectivity of the likelihood. Increasing $J$ reduces sampling noise and
leads to a visibly higher density of particles inside the failure domain. While both
gradient-based and gradient-free ALDI capture the geometry of the rare-event set,
the gradient-based variant produces a sharper concentration near the boundary
$\{G=0\}$, particularly for smaller ensemble sizes.

\smallskip
\begin{figure}[h!]
    \centering
    \begin{overpic}[width=\textwidth]{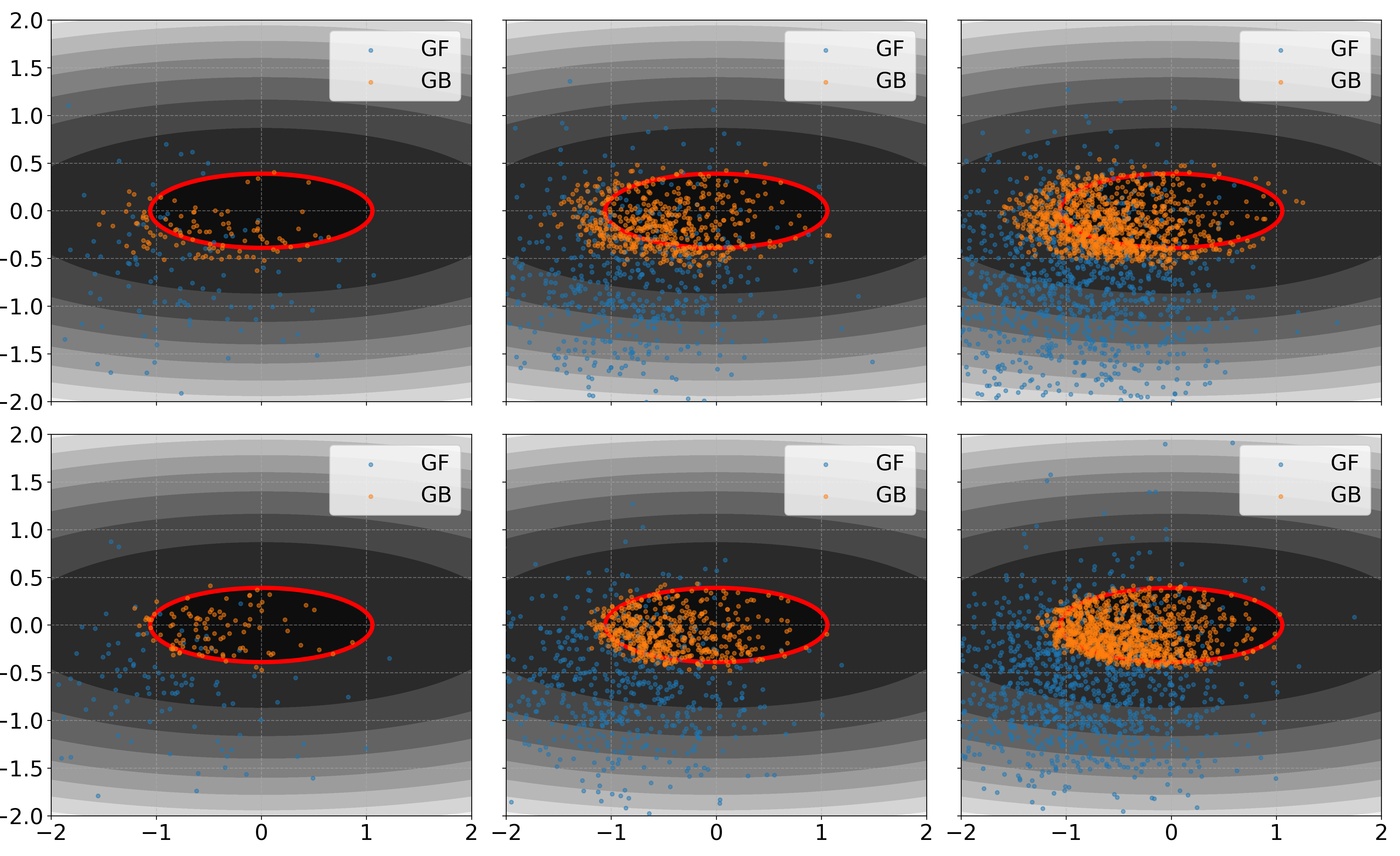}
   \put(15,61){\footnotesize$J=100$}
           \put(48,61){\footnotesize$J=500$}
           \put(80,61){\footnotesize$J=1000$}
           \put(-10,45){\footnotesize $R=0.1$}
           \put(-10,15){\footnotesize $R=0.01$}
    \end{overpic}
    \caption{ALDI ensembles for the hyperbolic saddle example for different ensemble sizes
    $J$ and noise levels $R$. The red ellipse indicates the failure set
    $\{A<r\}$, that is the set $\{G=0\}$}
    \label{fig:saddle_scatter}
\end{figure}

The number of particles located inside the failure domain is shown in
Figure~\ref{fig:saddle_mix}(a). As expected, this number increases with $J$ and is larger for smaller values of $R$, reflecting the stronger concentration of the ALDI
invariant measure near the failure ellipse. The effect is particularly pronounced for
the gradient-based dynamics, which generate more informative proposals for the
subsequent IS step.

Finally, Figure~\ref{fig:saddle_mix}(b) illustrates the dependence of the estimator
accuracy on the number of Gaussian mixture components used in the proposal.
For sufficiently large ensembles, increasing the number of components does not improve substantially the approximation of the failure probability as the error remains relatively constant.

\begin{figure}[h!]
    \centering
    \begin{subfigure}[t]{0.45\textwidth}
        \centering
         \centering
        \begin{overpic}[width=\textwidth]{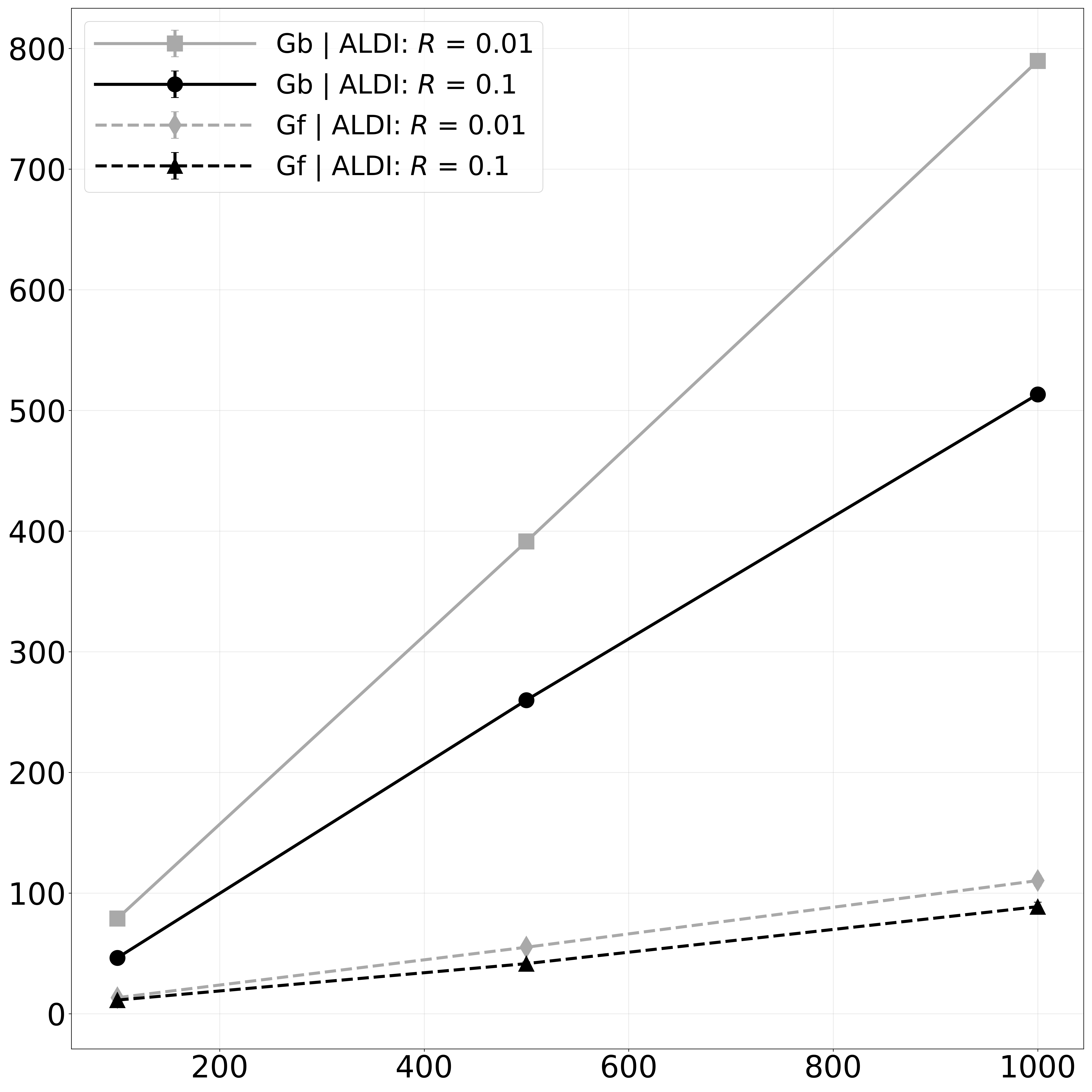}
        \put(50,-2){\footnotesize$J$}
        \put(-4,30){\footnotesize \begin{turn}{90} $\#$ particles in $F$\end{turn}}
        \end{overpic}
        \caption{}
    \end{subfigure}
    \hfill
    \begin{subfigure}[t]{0.45\textwidth}
        \centering
        \begin{overpic}[width=\textwidth]{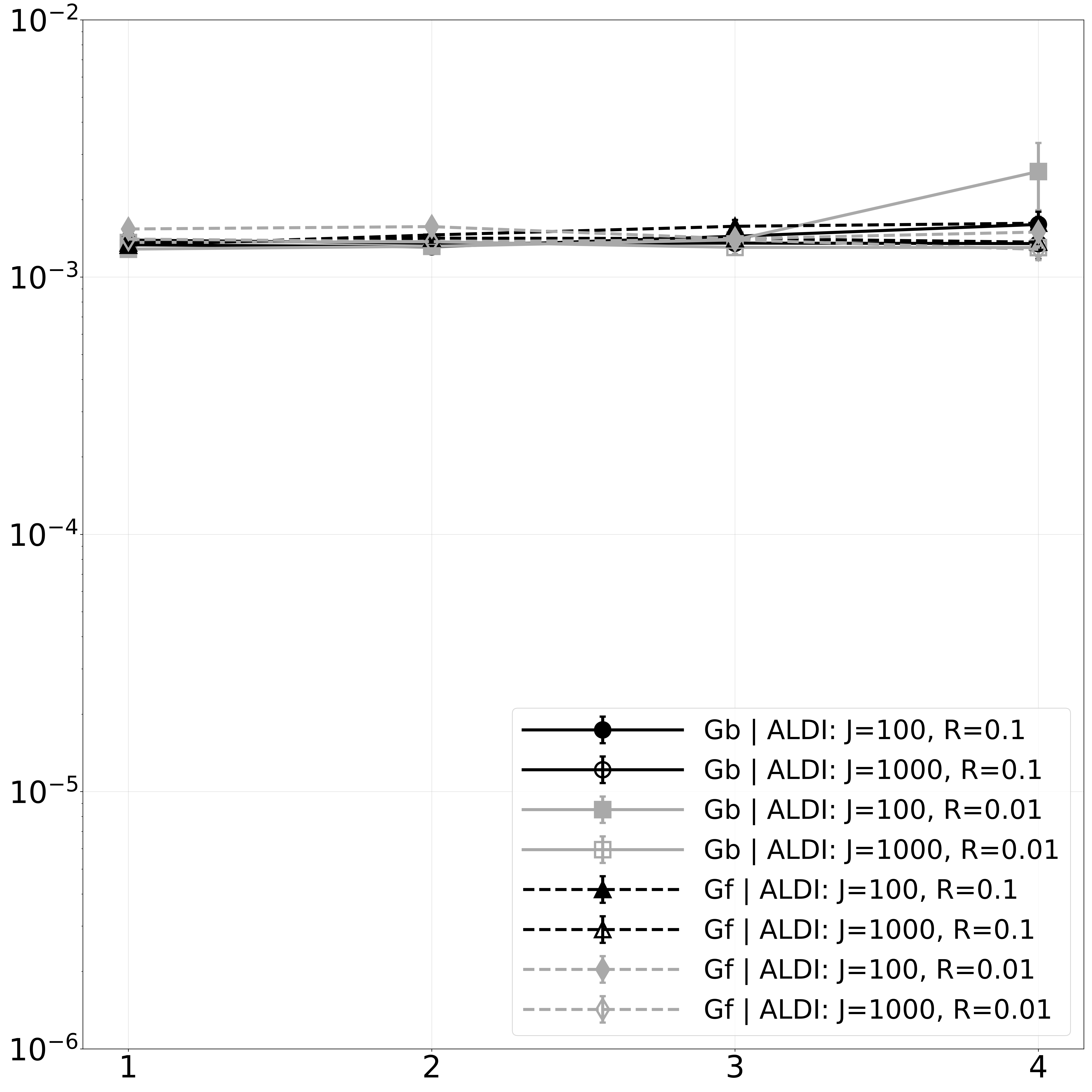}
        \put(20,-2){\footnotesize$\#$ Gaussian mixture components}
        \put(-10,40){\footnotesize \begin{turn}{90} $\frac{\lvert \widehat P_f^{\mathrm{IS}}- P_{\mathrm{ref}}\rvert}{P_{\mathrm{ref}}}$\end{turn}}
    \end{overpic}
        \caption{}
    \end{subfigure}
    \caption{In (a), the number of particles lying inside of the failure domain is shown. In (b), the relative error {of the estimator $\widehat{P}_f^{\mathrm{IS}}$} as a function of the number of Gaussian mixture
    components used in the importance-sampling proposal Alg.~\ref{alg:aldi_mixture_is} with $10^3$ samples.}
    \label{fig:saddle_mix}
\end{figure}

We next examine the accuracy of the importance-sampling estimator obtained from
the ALDI-generated proposals. Figure~\ref{fig:is_error_saddle} reports the relative error of $\widehat P_f^{\mathrm{IS}}$ as a function of the number of IS samples. For all ensemble
sizes, the error decays at the expected Monte Carlo  rate, confirming that the fitted
Gaussian mixtures provide effective proposal distributions. Furthermore, the variance decays also at the expected Monte Carlo rate.

\begin{figure}[h!]
    \centering
    \begin{subfigure}[t]{.45\textwidth}
    \centering
    \begin{overpic}[width=\textwidth]{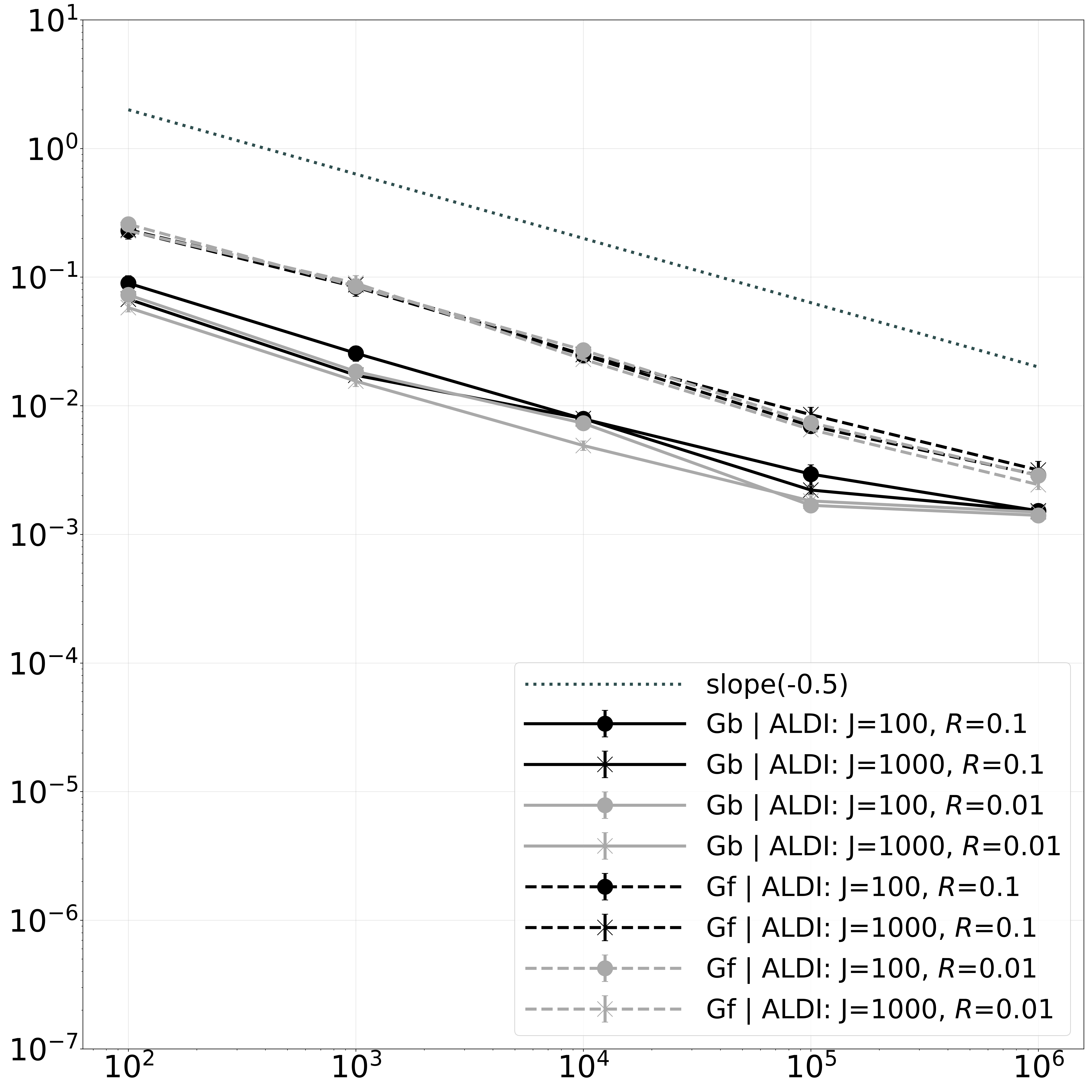}
        \put(-12,35){\footnotesize \begin{turn}{90}
        $\frac{\lvert \widehat P_f^{\mathrm{IS}}- P_{\mathrm{ref}}\rvert}{P_{\mathrm{ref}}}$
        \end{turn}}
        \put(40,-2){\footnotesize $\#$ IS samples}
    \end{overpic}
    \caption{}
    \end{subfigure}
    \hfill
    \begin{subfigure}[t]{.45\textwidth}
    \centering
    \begin{overpic}[width=\textwidth]{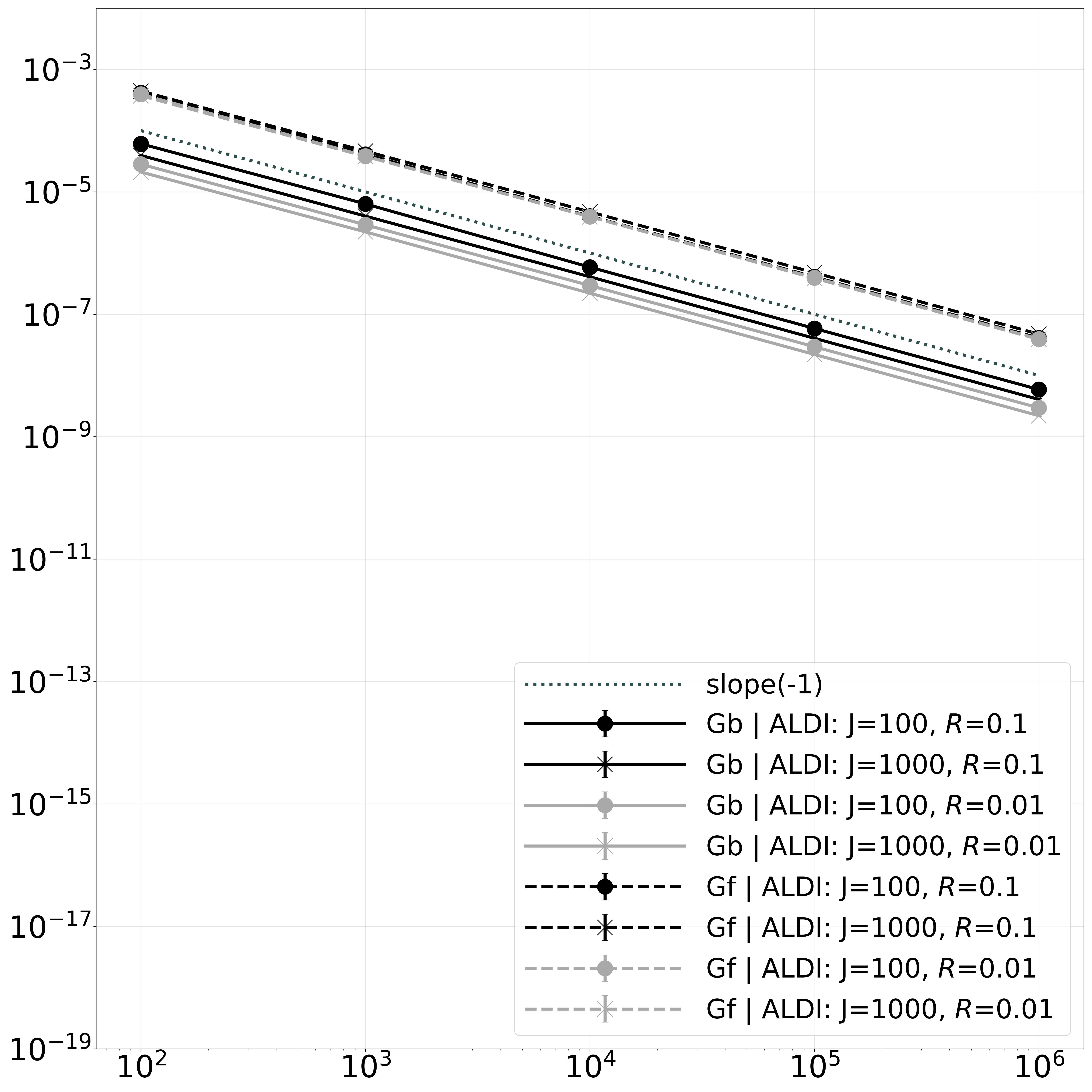}
        \put(-12,35){\footnotesize \begin{turn}{90}        $Var\left(\widehat{P}^{\mathrm{IS}}_f\right)$
        \end{turn}}
        \put(40,-2){\footnotesize $\#$ IS samples}
    \end{overpic}
    \caption{}
    \end{subfigure}
    \caption{The relative error and the variance of the IS estimator in Alg.~\ref{alg:aldi_mixture_is} as a function of
    the number of IS samples for $J=100,1000$ and $R=0.1,0.001$ are portrayed, respectively, for the hyperbolic saddle problem.    
   Both GB and GF ALDI follow a similar trend as the analytical predictions of crude Monte Carlo.}
    \label{fig:is_error_saddle}
\end{figure}

\paragraph{Summary}
Across both test cases, the numerical results are in good agreement with the
theoretical predictions. ALDI yields accurate estimates of the failure
probability, with errors decreasing as the ensemble size $J$ increases. Once
the smoothing parameter $\delta$ is sufficiently small, the estimator is largely
insensitive to its precise value, whereas smaller values of $R$ sharpen the
auxiliary distribution around the rare-event region and improve localization of
the failure set. Consistently across both examples, the number of particles
lying inside the failure domain increases with $J$ and is higher for smaller
$R$, indicating a better-adapted proposal for subsequent importance sampling.

Regarding the two estimation strategies, Alg.~\ref{alg:aldi_product} provides a useful
conceptual link between the biased ALDI measure and the original failure
probability through a product representation. However, the numerical
experiments confirm that directly estimating the associated normalisation
constant is unstable in practice and sensitive to incomplete stationarity of
the ALDI dynamics. In contrast, Alg.~\ref{alg:aldi_mixture_is} avoids this difficulty by fitting
an explicit proposal distribution to the ALDI ensemble and applying standard
importance sampling. This approach yields robust and systematically improving
estimates as $J$ increases, and its performance is directly correlated with the
quality of the ALDI-generated samples.

Overall, the two examples demonstrate that ALDI provides a reliable mechanism
for constructing effective proposals for rare-event estimation across problems
of different geometric and dynamical structure, while Alg.~\ref{alg:aldi_mixture_is} offers a
practically stable and accurate estimator in regimes where direct normalisation
is infeasible.

\subsection{Atmospheric blockings and rare events}
\label{SUBSEC:vortex}

An \textit{atmospheric blocking} is a quasistationary structure in the atmospheric flow over a large region. Their timescale ranges from several days to weeks, and can have a drastic impact in the local weather \cite{Kuhlbrodt00}. It becomes then of crucial relevance to estimate the probability of duration of an atmospheric blocking to remain over a long time window. Events of this nature are often captured by means of a real-valued function $G$  \cite{Ragone18,Ragone20}.

\subsubsection{Point-vortex dynamics}
A simple description of atmospheric flows at synoptic and planetary scales is given by the standard planar Euler equations, where the flow is assumed to be inviscid and incompressible \cite{Muller15}. Its expression in vorticity form is thus
\begin{equation}
\label{eq:Euler_vorticity}  
    \frac{\partial \bm{\omega}}{\partial t}+\bm{u}\cdot\nabla\bm{\omega}=\bm{\omega}\cdot \nabla \bm{u},
\end{equation}
where $\bm{u}(t,x)$ and $\bm{\omega}(t,x)$ are the velocity and vorticity fields of the flow on $x\in\mathbb{R}^2$ at time $t\geq0$, respectively.

Furthermore, assuming the vorticity is concentrated in $N$ disjoint points $X_j\in\mathbb{R}^2$, i.e.
\[
    \bm{\omega}(t,x)=\sum_{j=1}^N \Gamma_j \delta_{X_j(t)}(x),
\]
where $\delta_X$ denotes the Dirac delta distribution supported on $X$, one can derive an ordinary differential equation for the position of the point-vortices $X_j=(x_j,y_j)$, which reads \cite{Newton01}
 \begin{equation}
    \label{eq:Nvortices_plane}
        \begin{split}
            \dot{x}_j & =-\frac{1}{2\pi} \sum_{i\neq j} \frac{\Gamma_i(y_j-y_i)}{l_{ij}^2} \\
            \dot{y}_j & = \frac{1}{2\pi} \sum_{i\neq j} \frac{\Gamma_i(x_j-x_i)}{l_{ij}^2},
        \end{split}
    \end{equation}
    where $l_{ij}=\Vert X_i-X_j \Vert$ are the \textit{intervortical distances}. The system \eqref{eq:Nvortices_plane} is referred to as the \textit{$N$-vortex problem}. 
    
\begin{remark}
The derivation of the point-vortex dynamics turns out to be a good approximation of \eqref{eq:Euler_vorticity}, whenever the initial vorticity $\bm{\omega}(0,x)$ is supported on $N$ disjoint regions $\Lambda_j^{\varepsilon}$, contained in circles of radius $\varepsilon$ centred on $X_j(0)$, for $j=1,\ldots, N$. Indeed, under suitable assumptions, it has been shown that the corresponding solutions of \eqref{eq:Euler_vorticity}, converge weakly as $\varepsilon\rightarrow0$ to those given by \eqref{eq:Nvortices_plane}. We refer to \cite[Theorem 2.1]{Marchioro93} for the precise formulation of the statement. More recently, a connection of a stochastically perturbed version of \eqref{eq:Nvortices_plane} and the deterministic Navier-Stokes equation has been established, in the large ensemble limit $N\rightarrow\infty$ \cite{Flandoli24}
\end{remark}

The $N$-vortex problem has a noncanonical Hamiltonian structure. Indeed, if $\bm{X}=(X_1,\ldots,X_N)$ the system  \eqref{eq:Nvortices_plane} is written as
\[
\dot{\bm{X}} = \mathbb{J} \nabla \mathcal{H} (\bm{X}), \qquad \mathbb{J} = \left[ \begin{array}{cc}
    0 & \textup{diag}(\Gamma_1^{-1},\ldots,\Gamma_N^{-1}) \\
    -\textup{diag}(\Gamma_1^{-1},\ldots,\Gamma_N^{-1}) & 0
\end{array} \right]
\]
where the energy function $\mathcal{H}$ is given by
\begin{equation}
    \label{eq:Energy_Nvortex}
    \mathcal{H}(\bm{X}) = -\frac{1}{4\pi}\sum_{
          i\neq j} \Gamma_i\Gamma_j \ln l_{ij},
\end{equation}
and the sum is taken over all pairs $i,j\in\{1,\ldots,N\}$. For a comprehensive exposition of the dynamics of \eqref{eq:Nvortices_plane}, we refer to \cite{Newton01}.

\subsubsection{Atmospheric blockings as relative equilibria of point-vortices}

In \cite{Kuhlbrodt00}, blocking events are described by the dynamics of low-order point-vortex systems, typically when $N\leq4$. The mechanism is as follows: whenever a vortex configuration is relatively stable, any passive vortex $X_{N+1}$ (i.e. a vortex with circulation $\Gamma_{N+1}=0$) remains secluded in a large area surrounding the position of the vortices $X_1,\ldots, X_N$ for large time windows. An escape from such region can happen due to an external forcing or noisy fluctuations, see Figure~\ref{fig:passive_vortices}.  When the blocking is constituted by two (quasi)stationary vortices, we say that a \textit{high-over-low blocking} is present, see panel (a) in Figure~\ref{fig:passive_vortices}. In the case that it consists of a three point-vortex configuration, we call it an \textit{$\Omega$ blocking}, see panel (b) in Figure~\ref{fig:passive_vortices}.

\begin{figure}
     \centering
     \begin{subfigure}[b]{0.45\textwidth}
         \centering
          \begin{overpic}[width=\linewidth]{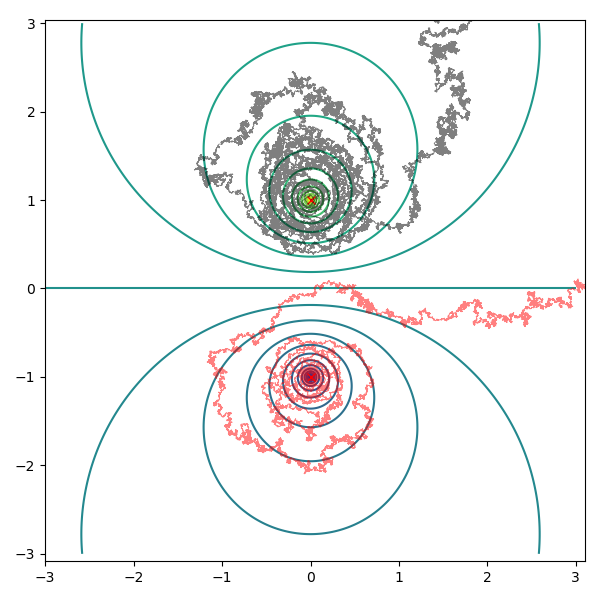}
           \put(50,-2){$x$}
      \put(3,45){$y$}
      \end{overpic}
      \caption{}
         \label{subfig:dipole_nonoise}
     \end{subfigure}
     \hfill
     \begin{subfigure}[b]{0.45\textwidth}
         \centering
          \begin{overpic}[width=\linewidth]{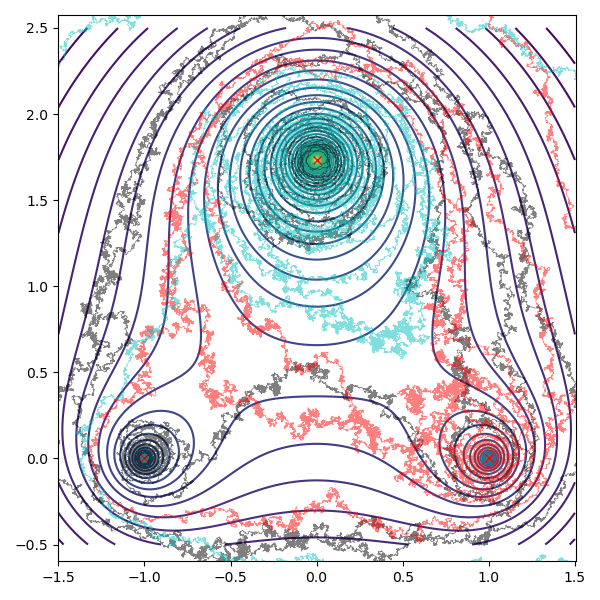}
           \put(50,-2){$x$}
      \put(3,45){$y$}
      \end{overpic}
      \caption{}
         \label{subfig:tripole_nonoise}
     \end{subfigure}
     \hfill
    
       \caption[Advected $N+1$-vortex problem]{Depiction of dipole and tripole configurations of point vortices in equilibrium. On panels (a) and (b), the energy landscapes for a passive vortex under the influence of two and three stationary point-vortices located at $X_j\in \mathbb R^2$ (marked as red crossed), respectively, are portrayed in solid lines. The Hamiltonian for the passive vortex $X=(x,y)$ under the influence of $N$ stationary vortices is given by $\mathcal{H}_{N+1}(x,y):= -\frac{1}{8\pi}\sum_{j=1}^N\ln\; \Vert X-X_j\Vert^2$, whose dynamics is perturbed with additive Brownian noise. Exemplary trajectories of such passive vortices, under the influence of additive Brownian noise, confirm the quasistationary behaviour in a bounded region induced by the vortex configuration.}
    \label{fig:passive_vortices}
\end{figure}

In the case of the 2-vortex problem, the dynamics is well understood - in the case that the total circulation $\Gamma=\Gamma_1+\Gamma_2$ is nonzero, both vortices $X_1,X_2$ rotate around the \textit{centre of vorticity} $X_c=(\Gamma_1/\Gamma) X_1 + (\Gamma_2/\Gamma) X_2$ with constant angular velocity $\Gamma/(2\pi\Vert X_1(0)-X_2(0)\Vert^2)$, see for instance \cite[p. 39]{Newton01}. In the case $\Gamma=0$, the vortices move in a parallel fashion, with speed $\sqrt{\Gamma}/(2\sqrt{2}\pi \Vert X_1(0)-X_2(0)\Vert)$. In any case, all possible dynamics are \textit{relative equilibria}, in the sense that the intervortical distance $l_{12}$ remains constant for all $t>0$.

\smallskip

\textit{Fixed equilibria} in the three-vortex problem, i.e. configurations in which $\dot{X}_i=0$ for $i=1,2,3$, are characterised by the positions being collinear and that the circulations satisfy $\Gamma_1\Gamma_2 + \Gamma_2\Gamma_3+ \Gamma_3\Gamma_1=0$. A configuration is a relative equilibrium, if and only if it is an equilateral triangle \cite[Theorem 2.2.1]{Newton01}. In the case $\Gamma:=\Gamma_1+\Gamma_2+\Gamma_3\neq 0$, the configuration rotates around the centre of vorticity with angular velocity $\Gamma/2\pi l^2$, where $l=l_{ij}$. When $\Gamma=0$, the configuration translates in parallel with velocity $(\Gamma_1^2+\Gamma_2^2+\Gamma_3^2)^{1/2}/(2\sqrt{2}\pi l)$.

From the observations above, it is clear almost none of the possible configurations in low-order point vortex dynamics admits fixed equilibria, and the pictures in Figure~\ref{fig:passive_vortices} may be misleading as a representation of the real phenomenon. However, the angular velocity depends on the total circulation $\Gamma$, in such a way that it vanishes as $\Gamma\rightarrow0$ or as the intervortical distances $l\rightarrow\infty$. Therefore, a quasi-stationary regime is expected at large scales.

\subsubsection{Stability of relative equilibria and rare events}
By definition, the space of relative equilibria $\mathcal{R}$ is an invariant manifold for the dynamics induced by \eqref{eq:Nvortices_plane}. In mathematical terms, the manifold $\mathcal{R}$ reads
\begin{equation}
    \label{eq:relative_equilibria_manifold}
\mathcal{R}= \left\{ \bm{X}\in\mathbb{R}^{2N} \; : \;  \frac{d l_{ij}}{dt}=0 \textup{ for } i,j=1,\ldots,N \right\}.
\end{equation}

As mentioned above, any solution of the 2-vortex problem is a relative equilibrium, and therefore $\mathcal{R}$ is the whole phase space $\mathbb{R}^{4}$. However, this is not the case for the 3-vortex problem. For instance, the set of equilateral triangular configurations $\mathcal{R}_\Delta$ is an invariant submanifold of $\mathcal{R}$ with $\textup{dim}\mathcal{R}_\Delta=4$. Indeed, the first vortices $X_1$ and $X_2$ are arbitrary points in $\mathbb{R}^2$, while the third one has only two possibilities, corresponding to the two possible orientations of the triangle $\Delta X_1 X_2 X_3$. Therefore, $\mathcal{R}_\Delta$ consists of two mutually disjoint branches.

The linear stability of $\mathcal{R}$ is given by the eigenvalues of its linearisation, where two of them are $0$, and two more are purely imaginary (possibly zero as well). Hence its stability is determined by the remaining two, which imply the following result.

\begin{theorem}
The invariant manifold $\mathcal{R}$ is unstable (of saddle type) if and only if \begin{equation}
    \gamma:=\Gamma_1\Gamma_2 + \Gamma_2\Gamma_3 + \Gamma_3\Gamma_1 <0. 
\end{equation}
In such a case, $\mathcal{R}_\Delta$ is a \textit{normally hyperbolic invariant manifold} of \eqref{eq:Nvortices_plane} with one-dimensional stable and unstable directions. More specifically, let $\varphi_t$ be the dynamical system induced by the 3-vortex problem. Then, there is a splitting of the tangent bundle
\[ T\mathbb{R}^6\vert_{\mathcal{R}_\Delta} = \mathcal{N}^s \oplus T\mathcal{R}_\Delta \oplus\mathcal{N}^u, \]
and local manifolds $W^{s,u}_{loc}(M)$ such that 

\smallskip

\begin{enumerate}
    \item $W^{s,u}_{loc}(\mathcal{R}_\Delta)$ is fibred as
    \[W^{s,u}_{loc}(\mathcal{R}_\Delta)= \bigcup_{p\in \mathcal{R}_\Delta} \mathcal{F}^{s,u}(p),\]
    where $\mathcal{F}^{s,u}(p)$ is a one-dimensional manifold based on $p\in\mathcal{R}_\Delta$,
    \item $\mathcal{F}^{s,u}(p)$ is tangent to $\mathcal{N}^{s,u}$ at $p$,
    \item The fibres $\mathcal{F}^{s}(p)$ and $\mathcal{F}^{u}(p)$ are positively and negatively invariant, respectively, that is for every $p\in\mathcal{R}_\Delta$ and $t\geq 0$,
    \begin{align*}
        \varphi_{t}\left( \mathcal{F}^s(p) \right) &\subset \mathcal{F}^s\left( \varphi_t(p)\right), \\
         \varphi_{-t}\left( \mathcal{F}^u(p) \right) &\subset \mathcal{F}^s\left( \varphi_{-t}(p)\right),
    \end{align*}
    \item There exist $C_{s,u}, \lambda_{s,u}>0$ such that if $q_{s,u}\in\mathcal{F}^{s,u}$, then
    \begin{align*}
        \Vert \phi_t(p)-\phi_t(q_s) \Vert < C_se^{-\lambda_st} \\
        \Vert \phi_{-t}(p)-\phi_{-t}(q_u) \Vert < C_ue^{-\lambda_ut}
    \end{align*}
\end{enumerate}
\end{theorem}

\begin{proof}
    We refer to \cite{Aref09} for the full analysis of the linear stability of relative equillibria. The structure of the tangent bundle into its stable and unstable parts is given by standard invariant manifold theory of dynamical systems, see for instance \cite{Wiggins94}.
\end{proof}
As observed in \cite{Aref09}, we have the following result.
\begin{corollary}
\label{CORO:equilibria_zeroGamma}
If $\Gamma=0$, any equilateral triangular configuration is an unstable relative equilibrium.
\end{corollary}
\begin{proof}
    It follows directly from the relationship $\Gamma^2 = (\Gamma_1^2+\Gamma_2^2+\Gamma_3^2) + 2\gamma$.
\end{proof}

In \cite{Muller15,Hirt18}, $\Omega$ blockings have been described by relative equilibria of three point vortices, whose circulations are given by $\Gamma_1=\Gamma_2$ and $\Gamma_3=-2\Gamma_1$, so that the total circulation $\Gamma$ vanishes. As stated in Corollary~\ref{CORO:equilibria_zeroGamma}, these configurations are unstable, of saddle type, so that small perturbations from such constellations deviate exponentially fast, unless the perturbed initial condition lies on $W^s_{loc}(\mathcal{R}_\Delta)$. Since $\textup{dim}\;W^s_{loc}(\mathcal{R}_\Delta)= 5$, it is a Lebesgue null set, and therefore this situation does not happen in practice. The association of atmospheric blockings to invariant and dynamically unstable objects has been also reported in \cite{Faranda16}. 

Instead, one can explore initial conditions which remain close to a relative equilibrium configuration for large time windows \cite{Ragone18,Ragone20}. In this case, given an energy value $H$, we consider the observable
\begin{equation}
    \label{eq:observable_releq}
    A(X_1,X_2,X_3):= \left\vert \cos\theta_1 -1/2\right\vert + \left\vert \cos\theta_2 -1/2\right\vert  + \left\vert\frac{ l_{12} + l_{23} + l_{31} }{3} -l(H)\right\vert,
\end{equation}
where $\theta_1$ is the angle between the vectors $X_2-X_1$ and $X_3-X_1$, $\theta_2$ is the angle betwwen the vectors $X_1-X_2$ and $X_3-X_2$, and $l(H)$ is the length of the unique equilateral configuration associated to the energy level $\mathcal{H}(X_1,X_2,X_3)=H$, that is $l(H)=e^{-\frac{4\pi H}{\gamma}}$. Clearly $A(\bm{X})\equiv A(X_1,X_2,X_3)=0$ if and only if $(X_1,X_2,X_3)\in\mathcal{R}_\Delta(H)$, i.e. if the points $X_1,X_2,X_3$ form an equilateral triangle of length $l(H)$. 

Additionally, we consider additive stochastic perturbations, as given by the SDE
\begin{equation}
    \label{eq:SDE_vortex}
d\bm{X}_t = \mathbb{J} \nabla\mathcal{H}(\bm{X}_t) + \sigma dW_t,
\end{equation}
with given initial condition $\bm{X}_0$. 
Here, $\sigma>0$ is the noise strength, and $W_t$ is a standard Brownian motion in $\mathbb{R}^6$. We denote the stochastic process induced by \eqref{eq:SDE_vortex} as $X_t(\omega)$. For a given time window $T>0$, let
\begin{equation}
    \label{eq:averaged_observable}
\bar{A}_T(\bm{X},\omega):= \frac{1}{T}\int_0^TA\left( \bm{X}_s(\omega) \right)ds.
\end{equation}
For $r>0$ sufficiently small, the rare event consists of the set
\[
    \left\{ (\bm{X}_0,\omega) : A_T(\bm{X}_s(\omega))<r \right\}.
\]
Therefore, our proposed limit-state function $G$ is defined as
\begin{equation}
    \label{eq:limit-state function_vortex}
G_r(\bm{X}):= \bar{A}_T(\bm{X})-r.
\end{equation}


\paragraph{Numerical results}
In the upcoming numerical explorations, we fix the circulations $\Gamma_1=\Gamma_2=1$ and $\Gamma_3=-2$. We fix the initial condition at an energy level $\mathcal{H}(\mathbf{X})=1$ on an equilateral triangle configuration of length $l(1)$. The threshold value used in this case is, $r=0.25$. We solve the $3$-vortex problem using the Euler-Maruyama method, with step size $\Delta T=0.02$ on a time horizon $T=0.5$. Additionally, the ALDI method is solved numerically using a step size $\Delta \tau=0.0005$ on an integration timescale $\tau=2.5$. In the importance sampling step, the number of Gaussian components used was 1. The reference probability, \(P_{\mathrm{ref}}=0.0592,\) was obtained by averaging five independent Monte Carlo runs of \(10^8\) samples.

Rare events are defined in terms of anomalous vortex configurations,
characterized by unusually persistent intervortical distances over a prescribed
time window. Figure~\ref{fig:vortex_structure} shows representative vortex
trajectories at different times, highlighting the emergence of dynamically
coherent structures associated with rare-event behaviour. While typical
configurations rapidly separate, rare-event trajectories remain clustered and
exhibit a markedly different geometric evolution. Although crude Monte Carlo sampling is prohibitively expensive in this regime,
a high-accuracy reference value $P_{\mathrm{ref}}$ is available from a
computationally intensive benchmark calculation and is used solely for
validation purposes.

\begin{figure}[h!]
    \centering
    \includegraphics[width=\textwidth]{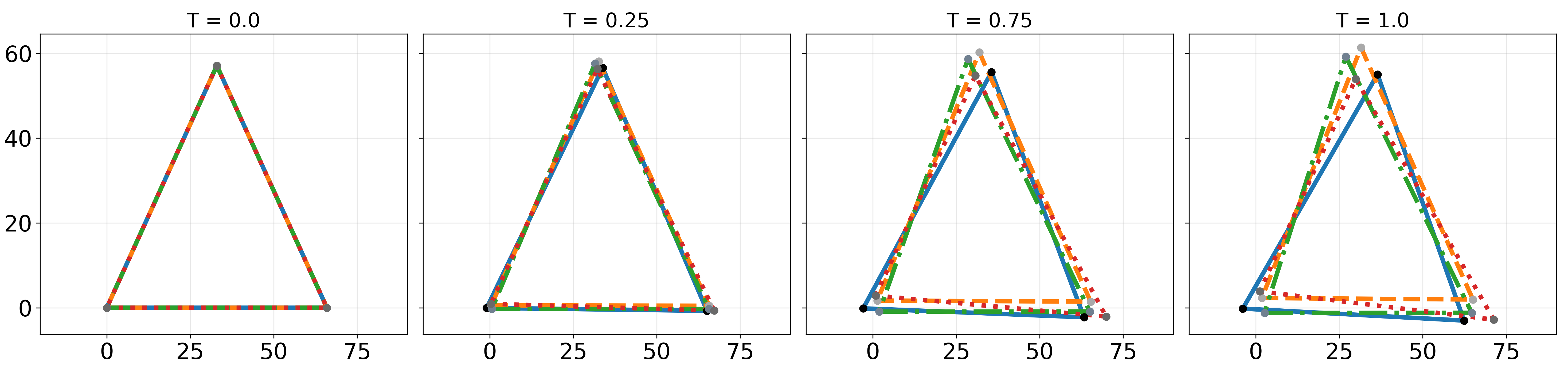}
    \caption{Evolution of representative three-vortex configurations in the plane at
    different times when the total circulation $\Gamma=\Gamma_1+\Gamma_2+\Gamma_3$ vanishes. The system starts on an equilateral configuration, which is already unstable in the absence of noise. Rare-event trajectories remain dynamically coherent,
    whereas typical trajectories exhibit rapid separation from the equilateral configuration.}
    \label{fig:vortex_structure}
\end{figure}

This behaviour is further quantified in
Figure~\ref{fig:vortex_distance_time}, which displays the evolution of
intervortical distances for different time horizons. For longer horizons,
typical trajectories show strong separation, whereas rare-event trajectories
remain close over extended periods, confirming the dynamical relevance of the
chosen failure event.

\begin{figure}[h!]
    \centering
        \begin{overpic}[width=0.65\textwidth]{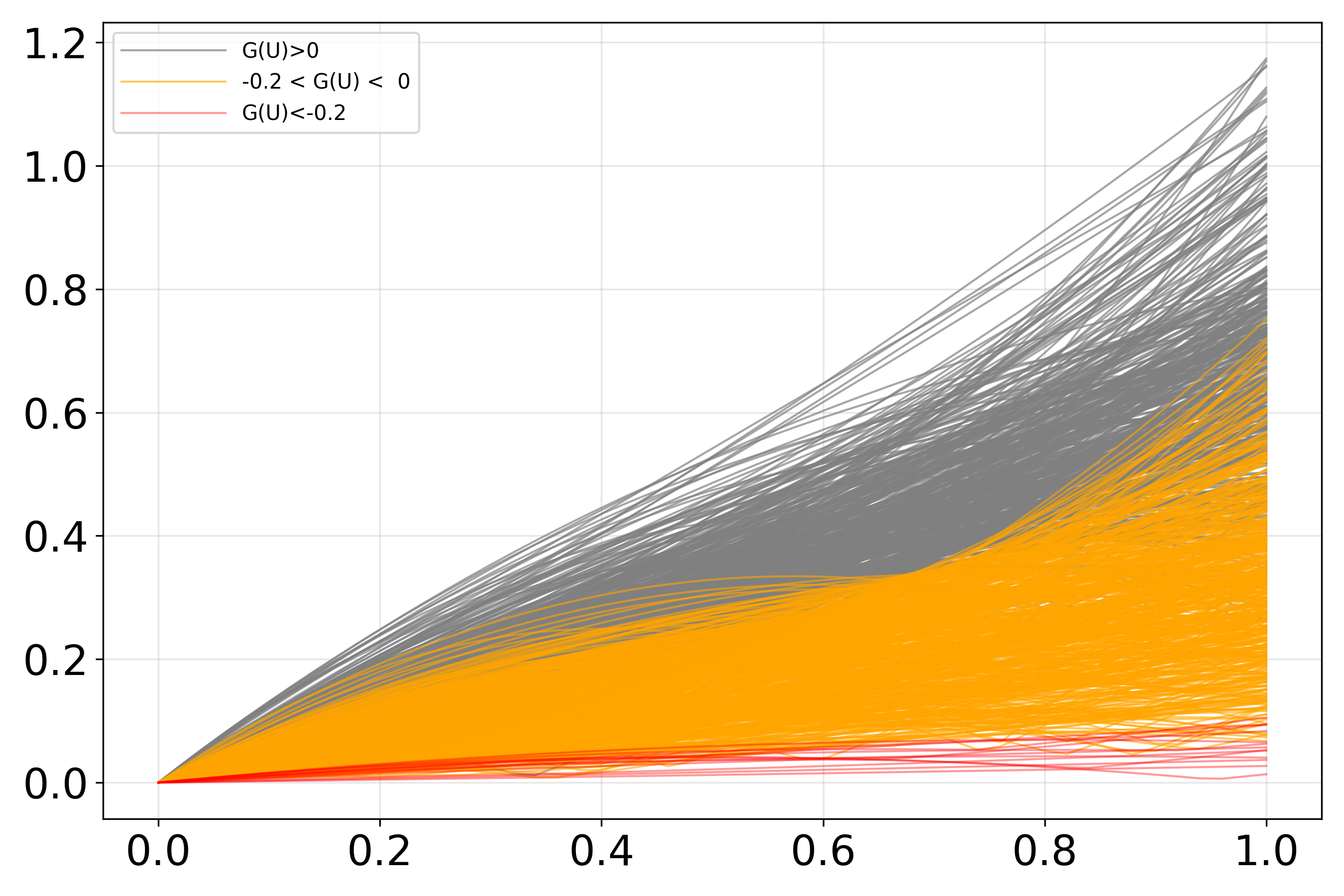}
        \put(50,0){\footnotesize$t$}
        \put(-10,35){\footnotesize$A(\mathbf{X}_t)$}
        \end{overpic}    
        
        \caption{The observable $A$ evaluated along the trajectories $X_t$ sampled using ALDI, as a function of time for different time. Here $J=1000$. Rare-event trajectories have relatively low value of $A$, compared to typical behaviour. The rare event sampled by ALDI corresponds to the trajectories portrayed in red and orange, corresponding to those for which $G\leq -0.2$ (13 particles) and $-0.2 \leq G\leq 0$ (553 particles), respectively, showing that the event is no longer rare for the ALDI samples.} 
    \label{fig:vortex_distance_time}
\end{figure}

Figure~\ref{fig:vortex_mix} summarizes the behaviour of the ALDI-generated
proposals and the subsequent mixture-based importance sampling step for the
point-vortex problem. Panel~(a) shows the number of particles lying inside the
failure domain as a function of the ALDI ensemble size $J$. As in the previous
examples, this number increases monotonically with $J$, indicating that larger
ensembles provide a more faithful exploration of the dynamically relevant rare
event region. Despite the strong sensitivity of the vortex dynamics to initial
conditions, ALDI consistently generates a non-negligible fraction of samples in
the failure set, which is essential for constructing effective IS proposals.

Panel~(b) of Figure~\ref{fig:vortex_mix} displays the absolute error of the
mixture-based estimator $\widehat P_f^{\mathrm{IS}}$ as a function of the number of Gaussian
mixture components used in the proposal. For sufficiently large ensembles, the
estimation error decreases as the mixture becomes more expressive, reflecting an
improved approximation of the ALDI sample cloud. For smaller ensembles, however,
the error exhibits increased variability when too many components are used,
highlighting the trade-off between model flexibility and the amount of available
fitting data.

\begin{figure}[h!]
    \centering
    \begin{subfigure}[t]{0.45\textwidth}
        \centering
         \centering
        \begin{overpic}[width=\textwidth]{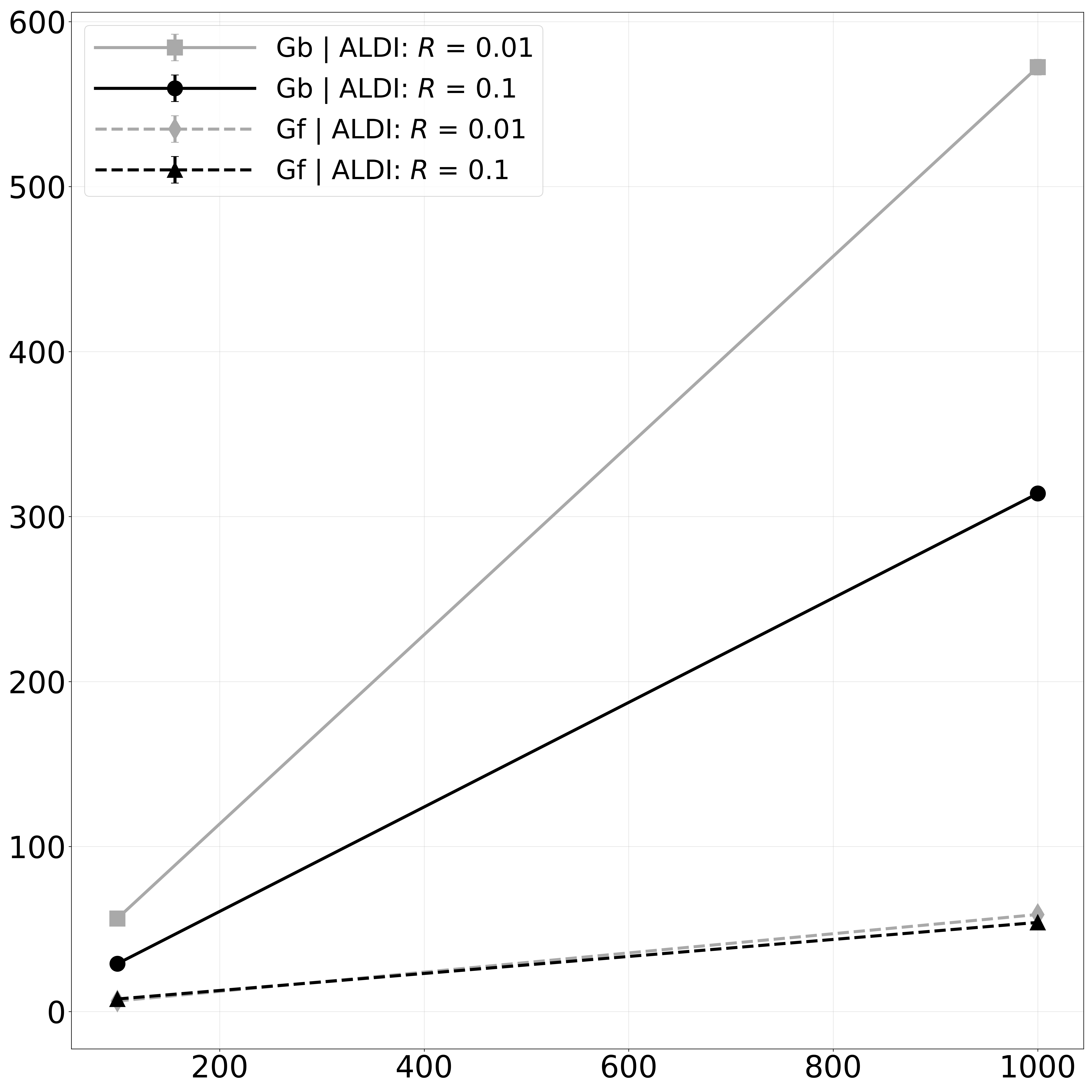}
        \put(50,-2){\footnotesize$J$}
        \put(-5,30){\footnotesize \begin{turn}{90} $\#$ particles in $F$\end{turn}}
        \end{overpic}
        \caption{}
    \end{subfigure}
    \hfill
    \begin{subfigure}[t]{0.45\textwidth}
        \centering
        \begin{overpic}[width=\textwidth]{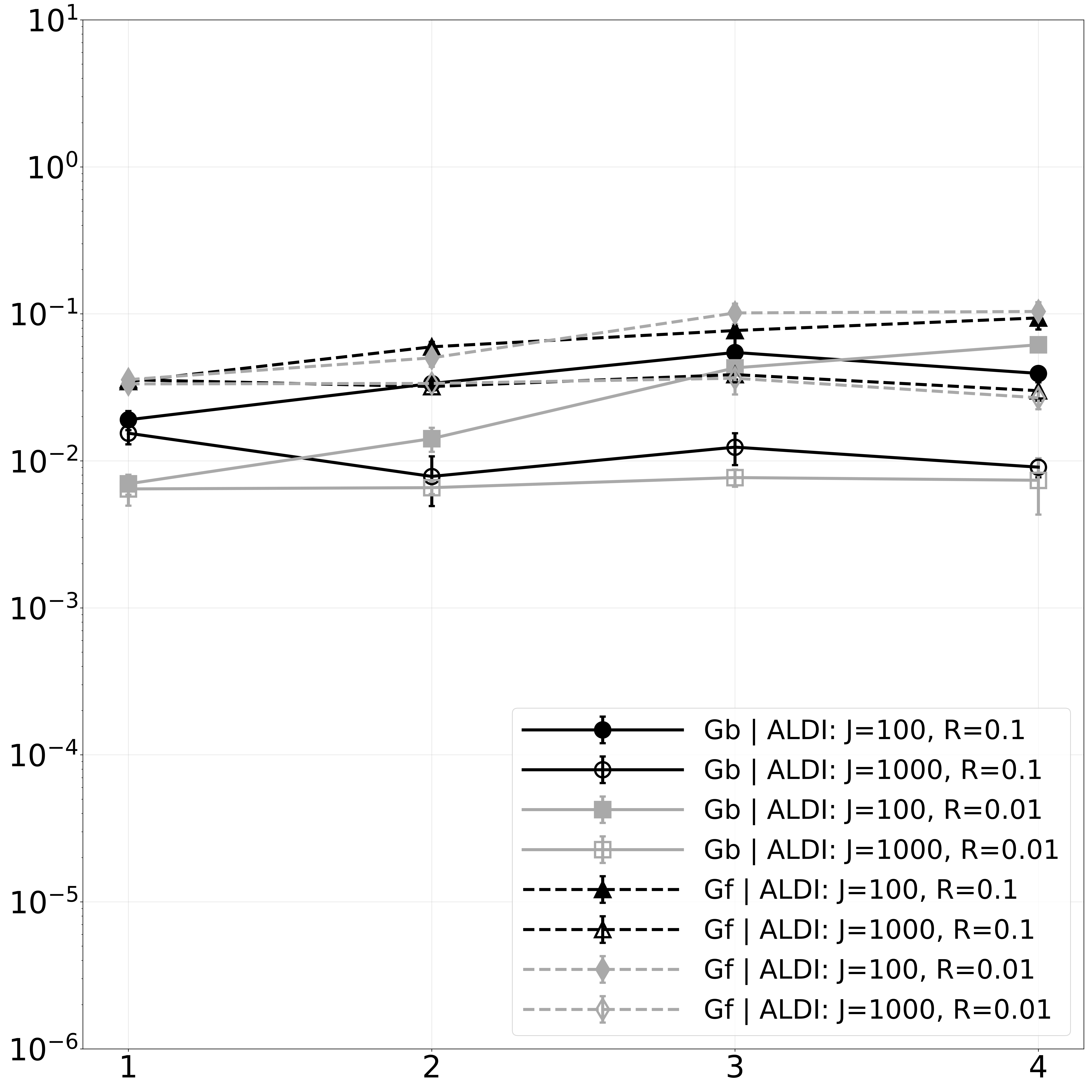}
        \put(15,-2){\footnotesize$\#$ Gaussian mixture components}
        \put(-12,35){\footnotesize \begin{turn}{90}
        $\frac{\lvert \widehat P_f^{\mathrm{IS}}- P_{\mathrm{ref}}\rvert}{P_{\mathrm{ref}}}$
        \end{turn}}
    \end{overpic}
        \caption{}
    \end{subfigure}
    \caption{For the vortex problem, the number of particles lying inside of the failure domain as a function of the ensemble size $J$ is shown in (a). In (b), the relative error {of the estimator $\hat{P}_f$} as a function of the number of Gaussian mixture
    components used in the importance-sampling proposal Alg.~\ref{alg:aldi_mixture_is} with $10^3$ samples is presented.}
    \label{fig:vortex_mix}
\end{figure}

The performance of the importance sampling correction is examined in
Figure~\ref{fig:is_error_vortex}, which reports the relative error of the IS
estimator as a function of the number of importance samples. For all ensemble sizes, the error decreases at the expected Monte Carlo  rate, confirming that the
mixture proposals fitted to the ALDI output yield stable and consistent IS
estimates. Larger ALDI ensembles lead to uniformly smaller errors, demonstrating
that improved localization of the rare-event region during the ALDI phase
directly translates into higher sampling efficiency in the IS step.

\begin{figure}[h!]
    \centering
    \begin{subfigure}[t]{.45\textwidth}
    \centering
    \begin{overpic}[width=\textwidth]{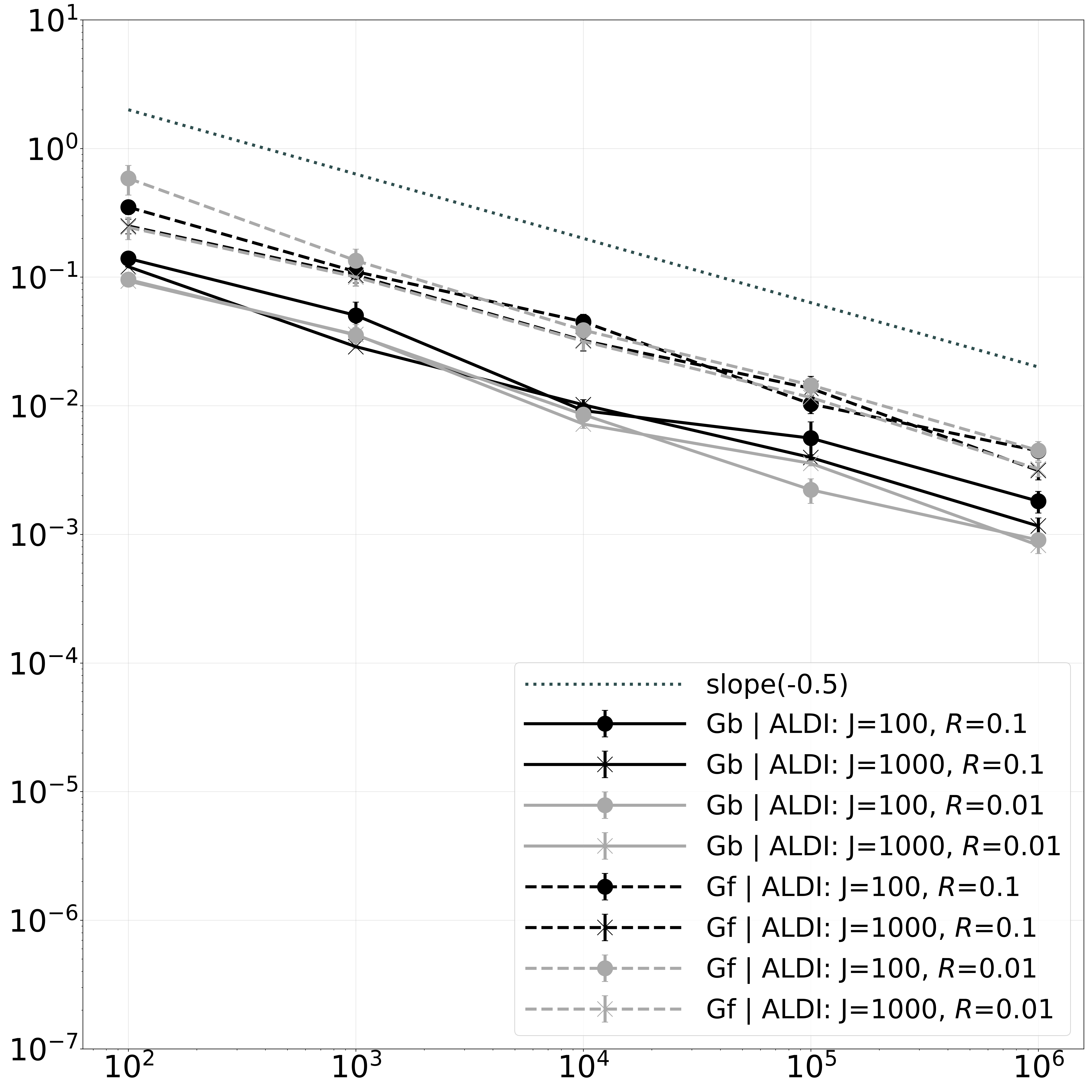}
        \put(-12,35){\footnotesize \begin{turn}{90}
        $\frac{\lvert \widehat P_f^{\mathrm{IS}}- P_{\mathrm{ref}}\rvert}{P_{\mathrm{ref}}}$
        \end{turn}}
        \put(40,-2){\footnotesize $\#$ IS samples}
    \end{overpic}
    \caption{}
    \end{subfigure}
    \hfill
    \begin{subfigure}[t]{.45\textwidth}
    \centering
    \begin{overpic}[width=\textwidth]{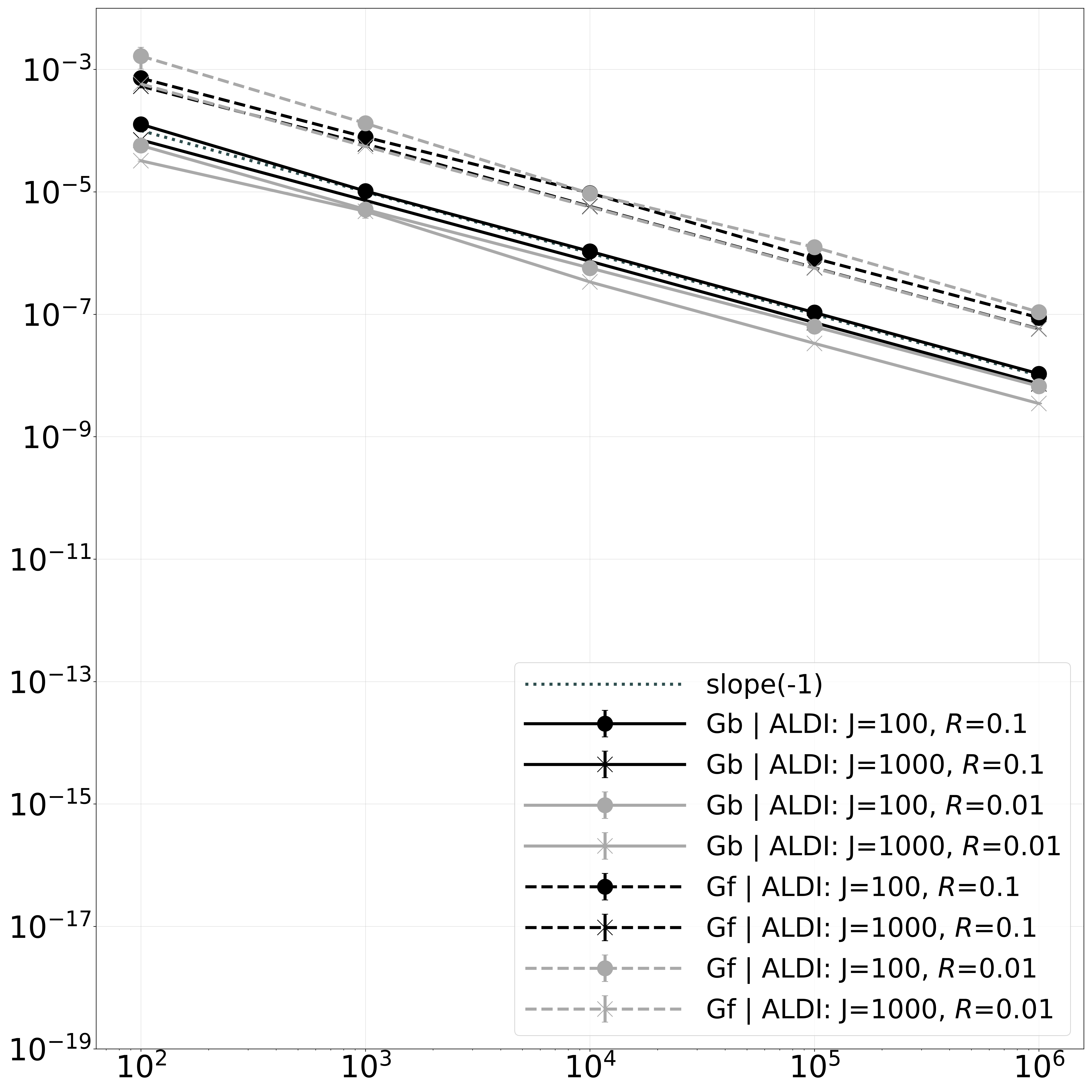}
        \put(-12,35){\footnotesize \begin{turn}{90}        $Var\left(\widehat{P}^{\mathrm{IS}}_f\right)$
        \end{turn}}
        \put(40,-2){\footnotesize $\#$ IS samples}
    \end{overpic}
    \caption{}
    \end{subfigure}
    \caption{The relative error and the variance of the IS estimator in Alg.~\ref{alg:aldi_mixture_is} as a function of
    the number of IS samples for $J=100,1000$ and $R=0.1,0.001$ are portrayed, respectively, for the point-vortex problem.    
    As in the previous examples, both GB and GF ALDI follow a similar trend.}
    \label{fig:is_error_vortex}
\end{figure}

\section{Conclusions}
\label{SEC:conclusions}
We have proposed an ALDI approach for the
efficient sampling of rare events. Our analysis shows that, after introducing a smooth
approximation of the limit-state function, the resulting posterior distribution
retains the correct small-noise limit and is amenable to sampling by ALDI. The
method naturally adapts to anisotropies
in the posterior induced by low-probability geometries.

The numerical experiments demonstrate that ALDI captures the qualitative and
quantitative features of rare-event sets across a hierarchy of models: from
convex and saddle-type finite-dimensional examples, and finally to point-vortex models for atmospheric
blockings. In all cases, the sampler concentrates around the relevant
near-critical manifolds and yields accurate proposal distributions for
self-normalised importance sampling.

Overall, ALDI provides an effective and robust tool for rare-event estimation in
complex dynamical systems, particularly in settings where gradients are
unreliable or unavailable and where the geometry of the event is highly
anisotropic. Future work includes quantitative convergence analysis of ALDI,
extensions to fully time-dependent PDEs and stochastic systems,
multi-level or adaptive variants that reduce the cost of sampling,
and combinations with likelihood-free or surrogate-based approaches
for high-dimensional geophysical applications.

\section*{Code availability}
The implementation of the methods and all scripts used to produce the results in this paper are publicly available at \cite{Chakraborty25}.

\section*{Acknowledgments}

The authors would like to thank Antonia Liebenberg for providing parts of the code used in the numerical experiments. The authors thank Robert Gruhlke for insightful discussions.

\bibliographystyle{siamplain}
\bibliography{references}
\end{document}